\documentclass[12pt]{amsart}
\usepackage{amsfonts,amssymb,latexsym,amsmath,amsthm,color}
\usepackage{enumerate}
\usepackage{fullpage}

\newtheorem*{hof*}{Hoffman's Bound}

\newtheorem{theorem}{Theorem}
\newtheorem{lemma}[theorem]{Lemma}
\newtheorem{corollary}[theorem]{Corollary}
\newtheorem{proposition}[theorem]{Proposition}

\theoremstyle{definition}
\newtheorem{definition}[theorem]{Definition}

\newtheorem{conjecture}[theorem]{Conjecture}

\theoremstyle{remark}

\newcommand{\hgq}[4]{
{}_{2}{\mathbb F}_{1} \left[
\begin{matrix}
#1 & #2\smallskip \\
   & #3
\end{matrix}
 ;\mbox{ }#4
\right]
}

\newcommand{\hgthree}[6]{
\,_{3}{\mathbb F}_{2} \left[
\begin{matrix}
#1 & #2 & #3 \\
   & #4 & #5 \\
\end{matrix}
\, ;\;\;\; #6
\right]
}

\newcommand*\HYPER{&}
\catcode`,\active
\newcommand*\pFFq{
\begingroup
\catcode`\,\active
\def ,{\HYPER}%
\doHyperF
}
\catcode`\,12
\def\doHyperF#1#2#3#4#5{%
\, _{#1}{\mathbb F}_{#2}\left[\begin{matrix}#3 \smallskip \\  #4\end{matrix} \; ; \; #5\right]%
\endgroup
}

\newcommand*\HYPERskip{&}
\catcode`,\active
\newcommand*\pFq{
\begingroup
\catcode`\,\active
\def ,{\HYPERskip}%
\doHyper
}
\catcode`\,12
\def\doHyper#1#2#3#4#5{%
\, _{#1}F_{#2}\left[\begin{matrix}#3 \smallskip \\  #4\end{matrix} \; ; \; #5\right]%
\endgroup
}

\def\C{\mathbb{C}}

\def\Q{\mathbb{Q}}
\def\F{\mathbb{F}}

\def\bk{\color{black}}

\newcommand\balpha{{\boldsymbol\alpha}}
\newcommand\bbeta{{\boldsymbol\beta}}

\begin{document}

\title[Intersecting Families of $PSL(2,q)$]{Characterization of intersecting families of maximum size in $PSL(2,q)$}

\author[Long, Plaza, Sin, and Xiang]{Ling Long$^*$, Rafael Plaza, Peter Sin$^{\dagger}$, Qing Xiang$^{\ddagger}$}

\thanks{$^*$Research partially supported by  NSF grants DMS-1303292 and DMS-1602047}
\thanks{$^{\dagger}$Research partially supported by a grant from the Simons Foundation (\#204181 to Peter Sin).}
\thanks{$^{\ddagger}$Research partially supported by an NSF grant DMS-1600850}

\address{Ling Long, Department of Mathematics, Louisiana State University, Baton Rouge, LA 70803, USA}
\email{llong@lsu.edu}

\address{Rafael Plaza, Department of Mathematical Sciences, University of Delaware, Newark, DE 19716, USA}
\email{plaza@udel.edu}

\address{Peter Sin, Department of Mathematics, University of Florida, Gainesville, FL 32611, USA} \email{sin@ufl.edu}

\address{Qing Xiang, Department of Mathematical Sciences, University of Delaware, Newark, DE 19716, USA} \email{qxiang@udel.edu}


\begin{abstract}

We consider the action of the $2$-dimensional projective special linear group $PSL(2,q)$ on the projective line $PG(1,q)$ over the finite field $\F_q$, where $q$ is an odd prime power.  A subset $S$ of $PSL(2,q)$ is said to be an intersecting family if for any $g_1,g_2 \in S$, there exists an element $x\in PG(1,q)$ such that $x^{g_1}= x^{g_2}$.  It is known that the maximum size of an intersecting family in $PSL(2,q)$ is $q(q-1)/2$.  We prove that all intersecting families of maximum size are cosets of point stabilizers for all odd prime powers $q>3$. 
\end{abstract}

%%%%%%%%%%%%%%%%%%%%%%%%%%%%%%%%%%%%%%%%%%%%%%

\maketitle

\section{Introduction}
%%%%%%%%%%%%%%%%%%%%%%%%%%%%%%%%%%%%%%%%%%%%%%

Let $n, k$ be positive integers such that $k\leq n$ and let $[n]=\{1,2,\ldots, n\}$.   A family of $k$-subsets of $[n]$ is said to be \emph{intersecting} if  the intersection of any two $k$-subsets in the family is non-empty.  The Erd{\H o}s-Ko-Rado (EKR) theorem is a classical result in extremal set theory. It states that when $k<n/2$ any intersecting family  of $k$-subsets has size at most ${n-1 \choose k-1}$; equality holds if and only if the family consists of all $k$-subsets of $[n]$ containing a fixed element of $[n]$ (cf. \cite{ekr}). In this paper, we focus on EKR type problems for permutation groups. In particular,  for any odd prime power $q$, we consider the natural right action of $PSL(2,q)$, the 2-dimensional projective special linear group over the finite field $\F_q$, on the set of points of $PG(1,q)$, the projective line over $\F_q$. 

Let $X$ be a finite set and $G$ a finite group acting on $X$. A subset $S$ of $G$ is said to be an {\it intersecting family} if for any $g_1,g_2 \in S$ there exists an element $x \in X$ such that $x^{g_1}= x^{g_2}$, i.e., $g_1g_2^{-1}$ stablizes some $x\in X$. In the context of  EKR-type theorems, the following problems about intersecting families in $G$  are  of interest:

\begin{enumerate}[I]
\item (Upper Bound) What is the maximum size of an intersecting family?
\item (Characterization) What is the structure of intersecting families of maximum size?
\end{enumerate}

Extensive research has been done  to solve the above problems for different groups.  In 1977, Deza and Frankl \cite{Frankl1} solved Problem I for the symmetric group $S_n$ acting on $[n]$. They proved that any intersecting family of $S_n$ has size at most $(n-1)!$. In fact,  this upper bound is tight because any coset of a point stabilizer in $S_n$ is an intersecting family of size precisely $(n-1)!$. They conjectured these sets are the only intersecting families of size $(n-1)!$.  This conjecture was proved to be true, independently, by Cameron and Ku \cite{Cameron1} and Larose and Malvenuto \cite{Larose1}.

In \cite{Karen1}, Meagher and Spiga studied Problem I and II for the group $PGL(2,q)$ acting on the set of points of the projective line $PG(1,q)$. These authors proved that the maximum size of an intersecting family in $PGL(2,q)$ is $q(q-1)$. Furthermore, they also solved the characterization problem: Every intersecting family of maximum size in $PGL(2,q)$ is a coset of a point stabilizer.  In \cite{Karen2}, they went one step further to solve Problem I and II for the group $PGL(3,q)$ acting on the set of points of the projective plane $PG(2,q)$. 

In this paper we study Problem II for the group $PSL(2,q)$ acting on $PG(1,q)$, where $q$ is an odd prime power. Here we only consider the $q$ odd case since if $q$ is a power of two, we have $PSL(2,q)=PGL(2,q)$, and both Problem I and II were solved in \cite{Karen1}.  It is known, from the combined results of \cite{Karen3, Karen1},  that the maximum size of an intersecting family in $PSL(2,q)$ is $q(q-1)/2$.  (In fact, in a recent paper \cite{mst}, it is proved that if $G\leq S_n$ is a 2-transitive group, then the maximum size of an intersecting family in $G$ is $|G|/n$. That is, the maximum size of an intersecting family is the cardinality of a point stabilizer.) However, it is only a conjecture that all intersecting families of maximum size are cosets of point stabilizers when $q>3$. (See the second part of Conjecture 1 in \cite{Karen1}.) In this paper, we prove that the second part of Conjecture 1 in \cite{Karen1} is true for all odd prime powers $q>3$. \bk

\begin{theorem}\label{psl_teo1}
Let $S$ be an intersecting family in $PSL(2,q)$ of maximum size, where $q>3$ is an odd prime power. Then  $S$ is a coset of a point stabilizer.  
\end{theorem} 

Note that when $q=3$, we have $PSL(2,q)\cong A_4$, and the action of $PSL(2,q)$ on the projective line $PG(1,q)$ is equivalent to the (natural) action of $A_4$ on $\{1,2,3,4\}$; in this case, it was pointed out in \cite{kuw} that the set $S=\{(1), (123), (234)\}$ (we are using cycle notation for permutations), is an intersecting family of maximum size in $A_4$, but $S$ is not a coset of any point stablizer. To prove Theorem \ref{psl_teo1} we apply a general method for solving Problem II for some $2$-transitive groups. This technique was described by Ahmadi and Meagher in \cite{Karen3} and they called it ``The Module Method''. This method reduces the characterization of intersecting families of maximum size to the computation of the $\mathbb{C}$-rank of a matrix which we define below.

\begin{definition}
Let $X$ be a finite set and $G$ a finite group acting on $X$.  An element $g\in G$ is said to be a \emph{derangement} if its action on $X$ is fixed-point-free. The {\it derangement matrix} of $G$ acting on $X$ is the $(0,1)$-matrix $M$, whose rows are indexed by the derangements of $G$, whose columns are indexed by the ordered pairs of distinct elements in $X$, and for any derangement $g \in G$ and $(a,b) \in X\times X$ with $a \neq b$, the $(g, (a,b))$-entry of $M$ is defined by
\[
M(g, (a,b)) = \left\lbrace \begin{array}{ll}
1, & \mbox{ if }a^g=b,\\
0, & \mbox{otherwise.}
\end{array} \right.
\] 
\end{definition}

The Module Method states that, under certain conditions, if the rank of the derangement matrix $M$ of $G$ acting on $X$ is equal to $(|X|-1)(|X|-2)$, then the cosets of point stabilizers are the only intersecting families of maximum size in $G$. This technique has been applied to show that the cosets of point stabilizers are the only intersecting families of maximum size for the symmetric group \cite{Karen4}, the alternating group \cite{Karen5}, $PGL(2,q)$ \cite{Karen1}, and many other groups \cite{Karen3}. 

Thus, in order to prove Theorem \ref{psl_teo1} by applying the Module Method, it is
enough to show that the rank of the derangement matrix $M$ of $PSL(2,q)$ acting on $PG(1,q)$ is equal to $q(q-1)$. Therefore, Theorem \ref{psl_teo1} follows directly from the next theorem.

\begin{theorem}\label{psl_teo2}
Let $M$ be the derangement matrix of $PSL(2,q)$ acting on $PG(1, q)$, where $q>3$ is an odd prime power. Then the $\mathbb{C}$-rank of $M$ is $q(q -1)$. %$q$ is big enough and $q-1$  is divisible by three, four or six then the $\mathbb{C}$-rank of $M$ is  $q(q -1)$.
\end{theorem}

Exactly the same statement for $PGL(2,q)$ is proved in \cite[Prop. 9]{Karen1}, 
so we must first  examine why the proof does not immediately carry over to
$PSL(2,q)$. In \cite{Karen1} the matrix $M^{\top}M$ represents a certain $PGL(2,q)$-module
endomorphism of a permutation module.  The main calculation is to show,
for each irreducible constituent character of this module, that 
the image of $M^{\top}M$ is not annihilated by the corresponding central idempotent.
Consequently, the image also contains the character as a constituent,
and the rank result follows due to the fact that the module in question is almost multiplicity-free, in the sense that, with one exception, each irreducible constituent character occurs with multiplicity one. If one attempts to follow the same
procedure for $PSL(2,q)$ one runs immediately into the problem that the 
$PSL(2,q)$-constituents of the permutation module have high  multiplicity. 
Fortunately, this obstacle can be sidestepped by observing that although
we are working in $PSL(2,q)$, our sets and permutation modules admit the action
of $PGL(2,q)$, and for the larger group the permutation module has the
property of being almost multiplicity-free. A more serious
difficulty arises when one attempts to show that the central idempotents
have nonzero images in the permutation module. As for $PGL(2,q)$, the problem
boils down to showing that certain sums of character values are not zero.
For $PGL(2,q)$, these sums could be estimated by elementary arguments.
However, the sums for $PSL(2,q)$ appear to be much harder to deal with, and
our proof proceeds by reformulating the sums
as character sums over finite fields and applying some deep results
on hypergeometric functions over finite fields. 
The finite field character sums which appear are Legendre and Soto-Andrade sums (see Section \ref{psl_ls_sums}). \bk This is not a surprise; it is well known that these sums appear in connection with the complex representation theory of $PGL(2,q)$ \cite{Kable}. To prove that these character sums are not equal to zero the following facts will be crucial:
\begin{enumerate}
	\item The Legendre and Soto-Andrade sums  (see Definitions \ref{def1} and \ref{def1_1})  on $\mathbb{F}_q$ form an orthogonal basis in the inner product space $\ell_2(\mathbb{F}_q,m)$ \cite{Kable}, where $m$ is the measure assigning mass $q+1$ to the points $\pm1$ and mass $1$ to all other points.
	\item The Legendre sums may be expressed in terms of hypergeometric functions over finite fields (see Section \ref{psl_hyp_sum}). These functions were introduced by Greene in \cite{Greene1}  and Katz in \cite{Katz1}  and since that time they have been extensively studied \cite{Ono1, LLong, Kable}.
\end{enumerate}

The rest of this paper is organized as follows. In Section 2,  we provide some  basic results about the character table of $PGL(2,q)$, Legendre and Soto-Andrade sums, and hypergeometric functions over finite fields.  In Section 3, we show that the rank of the derangement matrix $M$ is equal to the dimension of the image of a $PGL(2,q)$-module homomorphism. We use this fact to reduce the problem of computing the rank of $M$ to  that of showing  some explicit  character sums over $PGL(2,q)$ are not equal to zero.  In Section $4$, we find some formulas to express those character sums over $PGL(2, q)$ in terms of Legendre and Soto-Andrade sums. In Section 5, we prove Theorem  \ref{psl_teo2}. In Section 6, we conclude with some remarks and open problems. 

%%%%%%%%%%%%%%%%%%%%%%%%%%%%%%%%%%
\section{Background}
%%%%%%%%%%%%%%%%%%%%%%%%%%%%%%%%%%
We start by recalling standard facts about the
groups $PGL(2,q)$ and $PSL(2,q)$ and their complex characters, introducing our notation in the process. We shall assume that the reader is 
familiar with the general terminology and basic results 
from the representation theory of finite groups over the complex field, as can be found
in many textbooks, and we shall use \cite{Serre} for specific references when
necessary.
\subsection{The groups $PGL(2,q)$ and $PSL(2,q)$}
Let $\mathbb{F}_q$ be the finite field of size $q$ and $\mathbb{F}_{q^2}$ its unique quadratic extension. We denote by $\mathbb{F}_q^*$ and $\mathbb{F}_{q^2}^*$ the multiplicative groups of $\mathbb{F}_q$ and $\mathbb{F}_{q^2}$, respectively. 
Let $GL(2,q)$ be the group of all invertible $2 \times 2$ matrices over $\mathbb{F}_q$
and $SL(2,q)$ the subgroup of all invertible $2 \times 2$ matrices  with determinant $1$.
The center $Z(GL(2,q))$ of $GL(2,q)$ consists of all non-zero scalar matrices 
and we define $PGL(2,q)= GL(2,q)/Z(GL(2,q))$ and $PSL(2,q)=SL(2,q)/\left (SL(2,q)  \cap Z(GL(2,q))\right )$. If $q$ is odd then $PSL(2,q)$ is a subgroup of $PGL(2,q)$ of index $2$, while if $q$ is even then $PGL(2,q)= PSL(2,q)$.

We denote by $PG(1,q)$ the set of $1$-dimensional subspaces of the space
$\mathbb{F}_q^2$ of row vectors of length 2. Thus, $PG(1,q)$ is a projective line over $\mathbb{F}_q$ and its elements are called projective points. An easy computation shows that $PG(1,q)$ has cardinality $q+1$. From the above definitions, it is clear that the
$GL(2,q)$-action on $\mathbb{F}_q^2$ by right multiplication induces a natural right 
action of the groups $PGL(2,q)$  and $PSL(2,q)$ on  $PG(1,q)$. 
The action of the subgroup $PSL(2,q)$ is $2$-transitive, that is,
given any two ordered pairs of distinct points there is a group element sending  
the first pair to the second.  The  action of $PGL(2,q)$ is {\it sharply
$3$-transitive}, that is, given any two ordered triples of distinct points
there is a unique group element sending the first triple to the second. 
\bk

%%%%%%%%%%%%%%%%%%%%%%%%%%%%%%%%%%
\subsection{The character table of $PGL(2,q)$}\label{psl_ct}
%%%%%%%%%%%%%%%%%%%%%%%%%%%%%%%%%%

We assume in this section and throughout this paper that $q$ is an odd prime power. We briefly describe the character table of $PGL(2,q)$. We refer the reader to  \cite{Pia1} for a complete study of the complex irreducible characters of $PGL(2,q)$.  We start by describing its conjugacy classes. By abuse of notation we will denote the elements of $PGL(2,q)$ by  $2\times 2$ matrices with entries from $\mathbb{F}_q$. 

First note that, the elements of $PGL(2,q)$ can be collected into four sets: The set consisting of the identity element only; the set consisting of the non-scalar matrices with only one eigenvalue in $\mathbb{F}_q$; the set consisting of matrices with two distinct eigenvalues in $\mathbb{F}_q$; and the set of matrices with no eigenvalues in $\mathbb{F}_q$. Recall that the elements of $PGL(2,q)$ are projective linear transformations so if  $\{x_1,x_2\}$ are eigenvalues of some $g \in PGL(2,q)$ then $\{ax_1,ax_2\}$ are also eigenvalues of $g$ for any $a  \in \mathbb{F}_q^*$. Hence, the eigenvalues of elements in $PGL(2,q)$ are defined up to multiplication by elements of $\mathbb{F}_q^*$. 

The identity of $PGL(2,q)$, denoted by $I$, defines a conjugacy class of size $1$. Every non-identity element of $PGL(2,q)$ having only one eigenvalue in $\mathbb{F}_q^*$ is  conjugate  to
\[ 
u= \left(\begin{array}{cc}
 1 & 1\\
 0 & 1
\end{array}\right).
\]    
The conjugacy class of $u$ contains $q^2-1$ elements. The elements having two distinct eigenvalues in $\mathbb{F}_q$ are  conjugate  to
\[
d_x=\left(\begin{array}{cc}
 x & 0\\
 0 & 1
\end{array}\right)
\]
for some $x \in \mathbb{F}^*_q\setminus\{1\}$. Moreover, $d_x$ and $d_y$ are conjugate if and only if $x=y$ or $x=y^{-1}$. The size of the conjugacy class containing $d_x$ is $q(q+1)$  for  $x \in \mathbb{F}^*_q\setminus\{\pm 1\}$ and $q(q+1)/2$ for $x=-1$. Finally, the elements of $PGL(2,q)$ with no eigenvalues in $\mathbb{F}^*_q$ are  conjugate  to 
\[
v_r=\left(\begin{array}{cc}
 0 & 1\\
 -r ^{1+q} & r + r^q
\end{array}\right)
\] 
for some $r \in \mathbb{F}^*_{q^2} \setminus \mathbb{F}^*_q$. The matrices $v_r$ have eigenvalues $\{r, r^q\}$. Hence, $v_{r_1}$ and $v_{r_2}$ lie in the same conjugacy class if and only if  $r_1 \mathbb{F}^*_q = r_2 \mathbb{F}^*_q$ or $r_1 \mathbb{F}^*_q = r_2^{-1} \mathbb{F}^*_q$. The size of the conjugacy class containing $v_r$ is $q(q-1)$ if $r \in \mathbb{F}^*_{q^2} \setminus(\mathbb{F}^*_q \cup i\mathbb{F}^*_q)$ and $q(q-1)/2$ if $r \in i\mathbb{F}^*_q$, where $i$ is an element of $\mathbb{F}^*_{q^2}\setminus\mathbb{F}^*_q$ such that $i^2 \in \mathbb{F}^*_q$.

The complex irreducible characters of $PGL(2,q)$ are described in Table \ref{table1}. They also come in four families. First the characters $\lambda_1$ and $\lambda_{-1}$ correspond to representations of degree $1$. Here $\lambda_1$ is the principal character and the values of $\lambda_{-1}$ depend on a function  $\delta$ which is defined as follows: $\delta(x)=1$ if $d_x \in PSL(2,q)$ and $\delta(x)=-1$ otherwise, similarly, $\delta(r)=1$  if $v_r \in PSL(2,q)$ and $\delta(r)=-1$ otherwise.  

Secondly,  the characters $\psi_1$ and $\psi_{-1}$ correspond to representations of degree $q$. The character $\psi_1$ is the standard character which is an  irreducible character of $PGL(2,q)$. Thus,  for every $g  \in PGL(2,q)$,  the value of $\psi_{1}(g)$ is equal to the number of projective points fixed by $g$ in $PG(1,q)$ minus $1$. The values of $\psi_{-1}$ depend on the function $\delta$ defined above.

The third family is known as the cuspidal characters of $PGL(2,q)$. They correspond to representations of degree $q-1$ and their values depend on multiplicative characters of $\mathbb{F}_{q^2}$. In fact, the label $\beta$ in Table \ref{table1} runs through all homomorphism $\beta: \mathbb{F}_{q^2}^*/ \mathbb{F}_{q}^* \rightarrow \mathbb{C}^*$ of order greater than $2$ up to inversion. Note that every $\beta$ corresponds to a unique multiplicative character of $\mathbb{F}_{q^2}$ which is trivial on $\mathbb{F}_{q}^*$. 

Finally,  the fourth family of irreducible characters is known as the principal series of $PGL(2,q)$. These characters correspond to representations of degree $q+1$ and their values depend on multiplicative characters of $\mathbb{F}_{q}$. In fact, the label $\gamma$ in Table \ref{table1}  runs through all the homomorphism $\gamma : \mathbb{F}_{q}^* \rightarrow \mathbb{C}^*$ of order greater than 2 up to inversion.

Throughout this paper we denote by $\Gamma$ and $B$ a fixed selection of characters $\gamma$ and $\beta$, as defined above, up to inversion of size $(q-3)/2$ and $(q-1)/2$, respectively. Therefore, the principal series and cuspidal irreducible characters of $PGL(2,q)$ are given by $\{\nu_{\gamma}\}_{\gamma \in \Gamma}$ and $\{\eta_{\beta}\}_{\beta \in B}$, respectively.

 \begin{table}
 \caption{Character table of $PGL(2,q)$}\label{table1}
 \begin{center}
  \begin{tabular}{ | c | c | c | c | c | c | c | }
    \hline
                                 &   $I$      &  $u$      &  $d_{x}$                                         & $d_{-1}$            & $v_{r}$                          & $v_{i}$ \\ \hline
      $\lambda_1$       &     $1$    &   $1$     &  $1$                                              & $1$                    & $1$                              & $1$ \\ \hline
      $\lambda_{-1}$   &     $1$    &   $1$     &   $\delta(x)$                                  & $\delta(-1)$       & $\delta(r)$                    &  $\delta(i)$  \\ \hline
      $\psi_1$              &     $q$    &   $0$     &   $1$                                             & $1$                    & $-1$                             &  $-1$\\ \hline
      $\psi_{-1}$          &     $q$    &   $0$     &   $\delta(x)$                                  & $\delta(-1)$       & $-\delta(r)$                   &  $-\delta(i)$\\ \hline
      $\eta_{\beta}$     &   $q-1$   &  $-1$    &   $0$                                             & $0$                    & $-\beta(r)- \beta(r^q)$ & $-2\beta(i)$  \\ \hline
      $\nu_{\gamma}$  &  $q+1$  &   $1$     &   $\gamma(x) + \gamma(x^{-1})$ & $2\gamma(-1)$ & $0$                               & $0$ \\ \hline
  \end{tabular}
\end{center}
 \end{table}

%%%%%%%%%%%%%%%%%%%%%%%%%%%%%%%%%%
\subsection{Hypergeometric functions over finite fields}\label{psl_hyp_sum}
%%%%%%%%%%%%%%%%%%%%%%%%%%%%%%%%%%

 A (generalized) hypergeometric function  with parameters $a_i,b_j$ is defined by 
$$\pFq{n+1}{n}{a_1&a_2&\cdots&a_{n+1}}{&b_1&\cdots&b_n}{x}=\sum_{k\ge 1} \frac{(a_1)_k\cdots (a_{n+1})_k}{(b_1)_k\cdots(b_n)_k}\frac{x^k}{k!},$$ where $a_0=1$ and for $k\ge 1$, $(a)_k=a(a+1)\cdots(a+k-1)$ is called the Pochhammer symbol.

Hypergeometric functions over finite fields were introduced independently by John Greene \cite{Greene1} and Nicholas Katz \cite{Katz1}. Note that  the two definitions  differ only in  a normalizing factor  for cases  related to our discussion.  

In this section and throughout this paper we denote by $\epsilon$ and $\phi$ the trivial and quadratic multiplicative characters of $\F_q$, respectively.   Also throughout this paper \bk we adopt the convention of extending multiplicative characters by declaring them to be zero at $0\in \F_q$.  For any multiplicative character $\gamma$, we use $\overline \gamma$ to denote its complex conjugation.  A Gauss sum of $\gamma$ is defined by $g(\gamma):=\sum_{x\in\F_q}\gamma(x)\theta(x)$ where $\theta$ is any nontrivial additive character of $\F_q$. Let $\gamma_0, \gamma_1, \gamma_2 $ be multiplicative characters of $\F_q$ and $x \in \F_q$. Greene defines the following finite field analogue of a hypergeometric sum
\begin{equation}\label{ecu13}
\hgq{\gamma_0}{\gamma_1}{\gamma_2}{x;q}:= \epsilon(x)\frac{\gamma_1\gamma_2(-1)}{q} \sum_{y \in \mathbb{F}_q} \gamma_1(y)(\gamma_2\gamma_1^{-1}) (1-y) \gamma_{0}^{-1}(1-xy).
\end{equation}

Since the seminal work of Greene and Katz a lot of work has been done on special functions over finite fields, in particular generalized hypergeometric functions. In this section,  we recall some definitions and results that we will use later in this paper.

 Following  Greene \cite{Greene1}, we introduce other $_{n+1}\F_n$ functions inductively as follows. For  multiplicative characters $A_0,A_1, \ldots, A_n$ and $B_1,\ldots, B_n$ of $\mathbb{F}_q$ and $x \in \mathbb{F}_q$,  define 
\begin{multline*}
\pFFq{n+1}{n}{A_0 & A_1 & \cdots & A_n}{ & B_1 & \cdots & B_n}{x ; q}
:= \\  \frac{A_nB_n(-1)}{q} \sum_{y \in \mathbb{F}_q}\mbox{}  
\pFFq{n}{n-1}{A_0 & A_1 & \cdots & A_{n-1}}{ & B_1 & \cdots & B_{n-1}}{x y  ; q}  A_n(y) \overline{A_n}B_n(1-y).
\end{multline*} 
See \S 4.4 of \cite{LLong} for a comparison among different versions of finite field hypergeometric functions.  

The following lemma is a generalization of Lemma 2.2 in \cite{Ono1}. 

\begin{lemma}\label{lemma15}
For any non-trivial multiplicative character $\gamma$  of $\mathbb{F}_q$,
\[
q\pFFq{4}{3}{\gamma & \gamma^{-1} & \phi & \phi}{ & \epsilon & \epsilon & \epsilon}{1 ; q}
=  \sum_{z \in \mathbb{F}_q} \phi(z) \hgq{\phi}{\phi}{\epsilon}{z;q} \hgq{\gamma}{\gamma^{-1}}{\epsilon}{z;q},
\] where $\phi(\cdot)$ denotes the quadratic character of $\F_q$.

\end{lemma}

\begin{proof}
The lemma follows from the recursive definition of $_{n+1}\F_n$.  First,
\begin{eqnarray*}
q\pFFq{4}{3}{\gamma & \gamma^{-1} & \phi & \phi}{ & \epsilon & \epsilon & \epsilon}{1 ; q} & = & \phi(-1) \sum_{x \in \mathbb{F}_q^*} \phi(x)\phi(1-x) \hgthree{\gamma}{\gamma^{-1}}{\phi}{\epsilon}{\epsilon}{ x ;q} \\ 
& = & \frac{1}{q} \sum_{x \in \mathbb{F}_q^*}  \sum_{y \in \mathbb{F}_q^*} \phi(x) \phi(1-x) \phi(y) \phi(1-y) \hgq{\gamma}{\gamma^{-1}}{\epsilon}{xy;q}.
\end{eqnarray*}
Now replacing $xy$ by $z$, 
\begin{equation*}
q\pFFq{4}{3}{\gamma & \gamma^{-1} & \phi & \phi}{ & \epsilon & \epsilon & \epsilon}{1 ; q}   = \frac{1}{q} \sum_{x \in \mathbb{F}_q^*}  \sum_{z \in \mathbb{F}_q^*} \phi(1-x)  \phi(1-z/x) \phi(z)\hgq{\gamma}{\gamma^{-1}}{\epsilon}{z;q}.
\end{equation*}Letting $w=1/x$ and using \eqref{ecu13}
 we get,
\begin{eqnarray*}
q\pFFq{4}{3}{\gamma & \gamma^{-1} & \phi & \phi}{ & \epsilon & \epsilon & \epsilon}{1 ; q}  & = & \sum_{z \in \mathbb{F}_q^*} \frac{1}{q} \sum_{w \in \mathbb{F}_q^*} \phi(1 - \frac{1}{w}) \phi(1-zw) \phi(z) \hgq{\gamma}{\gamma^{-1}}{\epsilon}{z;q}  \\
& = & \sum_{z \in \mathbb{F}_q^*} \frac{1}{q} \sum_{w \in \mathbb{F}_q^*} \phi(-1)\phi(w)\phi(1-w) \phi(1-zw) \phi(z) \hgq{\gamma}{\gamma^{-1}}{\epsilon}{z;q}  \\
 & = &  \sum_{z \in \mathbb{F}_q}\phi(z) \hgq{\phi}{\phi}{\epsilon}{z;q} \hgq{\gamma}{\gamma^{-1}}{\epsilon}{z;q}. 
\end{eqnarray*}

\end{proof}

Like their classical counterparts hypergeometric functions over finite fields satisfy many transformation formulas \cite{LLong, Greene1}. In particular, the next one will be useful for our purpose.

\begin{lemma}\label{lemma12}
(Greene, \cite{Greene1}) For $x \in \mathbb{F}_q$ with $x \neq 0$ we have,
\[
\hgq{\phi}{\phi}{\epsilon}{x;q} = \phi(x) \hgq{\phi}{\phi}{\epsilon}{\frac{1}{x};q}. 
\]
\end{lemma}

\begin{proposition}\label{prop:15} Let  $n=2,3,4$ or $6$, $\F_q$ be any finite field of size $q$ that is congruent to $1\mod n$, and $\gamma$ be any  order $n$ multiplicative character of $\F_q$. Then  $$\left |q^3 \cdot \pFFq{4}{3}{\gamma&\gamma^{-1}&\phi&\phi}{&\epsilon &\epsilon &\epsilon}{1;q}+\phi(-1)\gamma(-1)q\right |\le 2q^{3/2}.$$ \end{proposition}
\begin{proof}This proposition is a corollary of  Theorem 2 of \cite{LTYW}. The background is about character sums in the perspective of hypergeometric motives  \cite{BCM,Katz1, RV2} and we will only point out how to obtain our claim.  Under the assumption on $n$, the choice of $\gamma$ is unique up to complex conjugation and $\gamma(-1)$ is independent of the choice of $\gamma$. For each $n$, let $\balpha=\{\frac 1n,\frac{n-1}n,\frac 12,\frac 12\}$ and $\bbeta=\{1,1,1,1\}$ and $\omega$ be any order $(q-1)$ multiplicative character of $\F_q$. Thus either $\gamma$ or $\gamma^{-1}$ is $\omega^{(q-1)/n}$. The normalized Katz version of hypergeometric sum is defined as (see  Definition 1.1 of \cite{BCM})

\begin{equation}\label{eq:H}
H_q(\balpha,\bbeta;\lambda):=\frac1{1-q}\sum_{k=0}^{q-2}
\prod_{\alpha\in \balpha}\frac{g(\omega^{k+(q-1)\alpha})}{g(\omega^{(q-1)\alpha})} \prod_{\beta\in \bbeta}\frac{g(\omega^{-k-(q-1)\beta})}{g(\omega^{-(q-1)\beta})}\,\omega^k\bigl((-1)^m\lambda\bigr).
\end{equation} We take $\lambda=1$ here. Then the conversion between Greene and the normalized Katz versions of hypergeometric finite sums says 
\begin{equation} -q^3 \cdot \pFFq{4}{3}{\gamma&\gamma^{-1}&\phi&\phi}{&\epsilon &\epsilon &\epsilon}{1;q}=H_q(\balpha,\bbeta;1),
\end{equation}independent of the choice of $\gamma$. Then Theorem 2 in \cite{LTYW} implies that there are two imaginary quadratic algebraic integers $A_{1,q}$ and $A_{2,q}$ (depending on both $n$ and $q$) both of complex absolute values $q^{3/2}$ such that $H_q(\balpha,\bbeta;1)=\phi(-1)\gamma(-1)q+A_{1,q}+A_{2,q}$. Thus $$\left |q^3 \cdot \pFFq{4}{3}{\gamma&\gamma^{-1}&\phi&\phi}{&\epsilon &\epsilon &\epsilon}{1;q}+\phi(-1)\gamma(-1)q\right |=|A_{1,q}+A_{2,q}|\le 2q^{3/2}.$$
The proof is now complete.
\end{proof}

%%%%%%%%%%%%%%%%%%%%%%%%%%%%%%%%%%
\subsection{The vector space $\ell^2(\mathbb{F}_q,m) $}\label{psl_ls_sums}
%%%%%%%%%%%%%%%%%%%%%%%%%%%%%%%%%%

Let $m: \mathbb{F}_q \rightarrow \mathbb{C}$ be $m(x)= 1 +q D_1(x) + q D_{-1}(x)$ where $D_{a}(x)$ is $1$ if $x=a$ and $0$ otherwise. We denote by   $\ell^2(\mathbb{F}_q,m) $ the vector space of complex-valued functions on $\mathbb{F}_q$ equipped with the Hermitian form 
\[
\langle f_1, f_2 \rangle :=\sum_{x \in \mathbb{F}_q} f_1(x) \overline{f_2(x)} m(x). 
\] 
Note that the following character sums are elements of  $\ell^2(\mathbb{F}_q,m)$.

\begin{definition}\label{def1}
For any multiplicative character $\gamma$ of $\mathbb{F}_q$, the \emph{Legendre sum} with respect to $\gamma$ is defined as
\[
P_{\gamma} (a) := \frac{1}{q} \sum_{x \in \mathbb{F}_q^*} \gamma(x)\phi(x^2-2ax+1), \quad \mbox{for all } a\in \mathbb{F}_q. 
\]
\end{definition}

\begin{definition}\label{def1_1}
For any multiplicative character $\beta$ of $\mathbb{F}_{q^2}$, the \emph{Soto-Andrade sum} with respect to $\beta$ is defined as
\[
R_{\beta} (a) := \frac{1}{q(q-1)} \sum_{r \in \mathbb{F}_{q^2}^*} \beta(r) \phi((r+r^q)^2 - 2(a+1)r^{1+q}), \quad \mbox{for all } a\in \mathbb{F}_q. 
\]
\end{definition}

The Legendre and Soto-Andrade sums have appeared several times in the literature in connection with the irreducible representations of $PGL(2,q)$ \cite{Kable}. In fact, we will encounter them in Section 4 in our study of some character sums over $PGL(2,q)$. In this section, we recall some properties of these sums that will be useful for us in the coming sections. 

The next lemma shows that the Legendre and Soto-Andrade sums form an orthogonal basis of $\ell^2(\mathbb{F}_q,m)$.

\begin{lemma}\label{lemma20}
(Kable, \cite{Kable}) The set 
\[
\mathfrak{L} :=\left\lbrace P_{\epsilon} - \frac{q-1}{q},  P_{\phi}, P_{\gamma}, R_{\beta} : \mbox{ } \gamma \in  \Gamma, \beta \in  B \right\rbrace
\]
is an orthogonal basis for the space $\ell^2(\mathbb{F}_q,m) $,  where $\Gamma$ and $B$ were defined in the end of Section \ref{psl_ct}  with $|\Gamma|=\frac{q-3}{2}$ and $|B|=\frac{q-1}{2}$. \bk The square norm of the elements of this basis are as follows:
\begin{eqnarray*}
 \left\Vert   P_{\epsilon} - \frac{q-1}{q} \right\Vert^2  & = &\frac{q^2-1}{q}, \\
 \Vert   P_{\phi} \Vert^2  & = &\frac{q^2-1}{q^2}, \\
 \Vert   P_{\gamma} \Vert^2  & = &\frac{q-1}{q}, \\
 \Vert   R_{\beta} \Vert^2  & = &\frac{q+1}{q}.
 \end{eqnarray*} 
\end{lemma}

If we normalize the basis given by Lemma \ref{lemma20} then we can easily obtain an orthonormal basis of $\ell^2(\mathbb{F}_q,m)$. We denote the elements of this orthonormal basis by $\{ P_{\epsilon}', P_{\phi}', P_{\gamma}', R_{\beta}' : \mbox{ } \gamma \in  \Gamma, \beta \in  B\}$.

The next lemmas list some elementary properties of the Legendre and Soto-Andrade sums that we will need later. Lemma \ref{lemma17} implies that the Legendre sum with respect to the trivial character is easy to evaluate.  This is not true for Legendre sums with respect to characters of higher orders. On the other hand, Lemma \ref{lemma18} shows that the Legendre and Soto-Andrade sums are easy to evaluate at $\pm 1$.  
\begin{lemma}\label{lemma17}
The values of the Legendre  sum with respect to $\epsilon$ are,
 \[
 P_{\epsilon}(a)= \left\lbrace  \begin{array}{ll}
 \frac{q-2}{q}, & \mbox{ if }a = \pm 1,\\
 -\frac{2}{q},  & \mbox{ if }a \neq \pm1.
 \end{array}\right. 
 \]
\end{lemma}
 
\begin{lemma}\label{lemma18}
Let $\gamma$ and  $\beta$ be characters from the sets $\Gamma$ and $B$, respectively. Then $P_{\gamma}(1) = -1/q$ and $R_{\beta}(1)= 1/q$. Moreover,
 \[
P_{\gamma}(-1) = - \frac{\gamma(-1)}{q},  \quad  R_{\beta} (-1) = -\frac{\beta(i)}{q} 
 \]
where $i \in \mathbb{F}_{q^2}^*\setminus \mathbb{F}_q$ such that $i^2 \in \mathbb{F}_q^*$. 
\end{lemma}
 
\begin{lemma}\label{lemma21}
The values of the Legendre and Soto-Andrade sums are real numbers. Moreover, for every $\gamma \in \Gamma$, $\beta \in B$ and $a \in \mathbb{F}_q$ we have  
 \[
P_{\gamma^{-1}} (a) = P_{\gamma}(a) \quad \mbox{ and } \quad  R_{\beta^{-1}} (a) = R_{\beta}(a).
 \]
\end{lemma}
 
The following result establishes a relation between Legendre sums and hypergeometric sums over finite fields. This fact will be crucial later in this paper.

\begin{lemma}\label{lemma13}
(Kable, \cite{Kable}) If $\gamma$ is a nontrivial character of $\mathbb{F}_q$ and $a \in \mathbb{F}_q\setminus \{\pm 1\}$ then 
\[
P_{\gamma} (a) = \hgq{\gamma}{\gamma^{-1}}{\epsilon}{\frac{1-a}{2};q}.
\]
\end{lemma}

%%%%%%%%%%%%%%%%%%%%%%%%%%%%%%%%%%
\section{A $PGL(2,q)$-module homomorphism}
%%%%%%%%%%%%%%%%%%%%%%%%%%%%%%%%%%

In this section we show that the rank of the derangement matrix $M$ of $PSL(2,q)$ is equal to the dimension of the image of a certain $PGL(2,q)$-module homomorphism. Actually, we will show that $N = M^{\top}M$ is a matrix representation of a $PGL(2,q)$-module homomorphism.  We will use this fact to compute the rank of $M$. 

%%%%%%%%%%%%%%%%%%%%%%%%%%%%%%%%%%
\subsection{The matrix $N$}
%%%%%%%%%%%%%%%%%%%%%%%%%%%%%%%%%%

We identify the points of the projective line $PG(1,q)$ with elements of the set $\mathbb{F}_q \cup \lbrace \infty \rbrace$, by letting  $a \in \mathbb{F}_q$ denote the point
spanned by $(1,a)\in \F_q^2$ and denoting by $\infty$ the point spanned by $(0,1)$. 
We consider the natural right action of $PGL(2,q)$ on $PG(1,q)$. Let $a \in \mathbb{F}_q \cup \lbrace \infty \rbrace$ and $g \in PGL(2,q)$.  We use $a^g$ to denote the element in $PG(1,q)$ obtained by applying $g$ to $a$. The action of $PGL(2,q)$ on $PG(1,q)$ is faithful.  Hence, we can associate with each element of $PGL(2,q)$ a permutation of the $q+1$ elements of $PG(1,q)$.  Moreover, recall that an element $g \in PGL(2,q)$ is said to be a {\it derangement} if its associated permutation is fixed-point-free.

\begin{definition}\label{def_N}
Let $\Omega$ be the set of ordered pairs of distinct projective points in $PG(1,q)$. The matrix $N$ is a $q(q+1)$ by $q(q+1)$ matrix whose rows and columns are both indexed by the elements of $\Omega$; for any $(a,b), (c, d) \in \Omega$ we define
\[
	N_{(a,b),(c,d)} {:=} \mbox{ the number of derangements of } PSL(2,q) \mbox{ sending } a \mbox{ to } b \mbox{ and } c \mbox{ to } d.
\]
\end{definition}

 Note that the above definition of $N$ agrees with our former definition, $N=M^{\top}M$. Hence, basic linear algebra implies that $\mbox{rank}_{\C}(M)=\mbox{rank}_{\C}(N)$.  The next lemma gives information about the entries of $N$.

%There is a very simple relation between the conjugacy classes of $PGL(2,q)$ and the number of points in $PG(1,q)$ fixed by the elements of $PGL(2,q)$:  the elements in the conjugacy class of $u$ fix one point, the elements in the conjugacy class of $d_x$ (for any $x \in \mathbb{F}_q^* \setminus \{1 \}$) fix two points and the elements in the conjugacy class of $v_r$ (for any $r \in \mathbb{F}_{q^2}^* \setminus \{ \mathbb{F}_q, i\mathbb{F}_q \}$) are derangements. 

\begin{lemma}\label{lemma1}
Let $a,b,c,d \in \mathbb{F}_q \cup \lbrace \infty \rbrace$. Then,
\begin{enumerate}
  \item $\displaystyle N_{(a,b),(a,b)} = \frac{(q-1)^2}{4}, \quad \forall (a,b) \in \Omega$.
  \item $N_{(a,b),(c,d)}= 0$, if $a=c, b \neq d$ or $a \neq c, b = d$.
  \item $N_{(a, b),(b, a)} = \left\lbrace \begin{array}{ll}
0, & \mbox{ if } q \equiv 1 \mbox{ mod } 4,\\
(q-1)/2, & \mbox{ if } q \equiv 3 \mbox{ mod } 4,
\end{array} \right. \quad \forall (a,b) \in \Omega $.

\item 
	\begin{enumerate}
	\item $N_{(0, \infty),(1,0)} = \left\lbrace \begin{array}{ll}
				(q-1)/4, & \mbox{ if } q \equiv 1 \mbox{ mod } 4,\\
				(q-3)/4, & \mbox{ if } q \equiv 3 \mbox{ mod } 4.
				\end{array} \right. $
	\item $\displaystyle  N_{(0,\infty),(1,d)} = \frac{q-3}{4} - \frac{\phi(1-d)}{2} - \frac{1}{4} \sum_{x \in \mathbb{F}_q^*} \phi((x + x^{-1})^2 - 4d), \quad \forall d \neq 0,1, \infty.$
	\end{enumerate}
\end{enumerate}
Moreover, the value of $N_{(a,b),(c,d)}$ for any $(a,b),(c,d) \in \Omega$ is given by one of the above expressions.
\end{lemma}

\begin{proof}

Let $g$ be an arbitrary element in $PGL(2,q)$. Note that for every  $h \in PSL(2,q)$ sending $a$ to $b$ and $c$ to $d$, the element  $g^{-1}h g \in PSL(2,q)$ sends $a^g$ to $b^g$ and $c^g$ to $d^g$. Hence the entries of $N$ satisfy the following property
\begin{equation}\label{ecu1}
N_{(a,b),(c,d)} = N_{(a^g,b^g),(c^g,d^g)},
\end{equation}
because $PSL(2,q)$ is a normal subgroup of $PGL(2,q)$ and the set of derangements in $PSL(2,q)$ is closed under conjugation.
To prove Lemma \ref{lemma1} we proceed case by case.

\begin{itemize}
\item \textbf{Case 1}.

\hspace{0.5cm}Recall that $N_{(a,b),(a,b)}$ is the number of derangements in $PSL(2,q)$ sending $a$ to $b$.  From Equation (\ref{ecu1}) and the $2$-transitivity of $PGL(2,q)$ we conclude that $N_{(a,b),(a,b)}= N_{(c,d),(c,d)}$ for any $(a,b), (c,d) \in \Omega$. The total number of derangements in $PSL(2,q)$ is $q(q-1)^2/4$ and this number can also be written as
\[
    \frac{q(q-1)^2}{4} = \sum_{\substack{b \in PG(1,q) \\ b \neq a} } N_{(a,b),(a,b)}, \quad \mbox{for any fixed }a \in PG(1,q),
\]
which implies that $N_{(a,b),(a,b)}= (q-1)^2/4$ for every $(a,b) \in \Omega$.

\item \textbf{Case 2}.

\hspace{0.5cm}Every element of $PSL(2,q)$ is related to a permutation of projective points in $PG(1,q)$. This implies $N_{(a,b)(a,d)}=0$ and $N_{(a,b)(c,b)}=0$ whenever $b \neq d$ and $a \neq c$.

\item \textbf{Case 3}.

\hspace{0.5cm}Using the $2$-transitivity of $PGL(2,q)$ and Equation (\ref{ecu1}) we can assume  without loss of generality  that $a=0$ and $b=\infty$. The elements $g_{\lambda} \in PSL(2,q)$ sending $0$ to $\infty$ and $\infty$ to $0$ are of the form
\[
g_{\lambda} := \left(  \begin{array}{cc}
0 & \lambda \\
-\lambda^{-1} & 0 
\end{array}\right), \quad \lambda \in \mathbb{F}_q^*. 
\]

\hspace{0.5cm}This representation of elements in $PSL(2,q)$ is redundant because  $g_{\lambda}$ and $g_{-\lambda}$ represent the same element of $PSL(2,q)$.  Let $\xi$ be an element in  $\mathbb{F}_q^* $ such that $\langle \xi\rangle = \mathbb{F}_q^* $. Hence, the set $\{ g_{\lambda} : \lambda=\xi^i, \quad i=1, \ldots, (q-1)/2\}$ corresponds precisely to the $(q-1)/2$ elements in $PSL(2,q)$ sending $0$ to $\infty$ and $\infty$ to $0$.

\hspace{0.5cm}Recall that $g_{\lambda}$ is a derangement if and only if its eigenvalues are not in $\mathbb{F}_q$. Thus, $g_{\lambda}$ is a derangement if and only if  its characteristic polynomial,
\[
p_{\lambda}(t) :=\det \left\vert \begin{array}{cc}
-t & \lambda\\
-\lambda^{-1} & -t 
\end{array}  \right\vert = t^2 + 1,
\]
is irreducible  over  $\mathbb{F}_q$.

\hspace{0.5cm} If $q \equiv 1 \pmod 4$ then $-1$ is a square in $\F_q$. Thus,  $p_{\lambda}(t)$ is reducible for every $\lambda \in \mathbb{F}_q^*$. Hence $N_{(a,b),(b,a)} =N_{(0,\infty),(\infty,0)} = 0$ in this case. On the other hand,  if $q \equiv 3 \pmod 4$ then $-1$ is not a square in $\F_q$. This implies that $p_{\lambda}(t)$ is irreducible for every $\lambda \in \F_q^*$. Therefore, $N_{(a,b),(b,a)} =N_{(0,\infty),(\infty,0)} = (q-1)/2$.

\item \textbf{Case 4}.

\hspace{0.5cm}Every element of $PSL(2,q)$ sending $0$ to $\infty$ and $1$ to $d$ is of the form
\[
g_{\lambda} :=\left(  \begin{array}{cc}
0 & -\lambda \\
\lambda^{-1} & \lambda^{-1}d + \lambda 
\end{array}\right), \quad \lambda \in \mathbb{F}_q^* .
\]
Again note that  $g_{\lambda}$ and $g_{-\lambda}$ represent the same element of $PSL(2,q)$. The matrix $g_{\lambda}$ is a derangement if and only if its characteristic polynomial,
\[
p_{\lambda}(t) :=\det \left\vert \begin{array}{cc}
-t & -\lambda\\
\lambda^{-1} & \lambda^{-1}d + \lambda -t 
\end{array}  \right\vert = t^2 - (\lambda^{-1}d + \lambda)t+1,
\]
is irreducible  over  $\mathbb{F}_q$. To compute $N_{(0,\infty) (1,d)}$ it is enough to count the number of values of $\lambda$ such that $p_{\lambda}(t)$ is reducible.

\hspace{0.5cm}If $p_{\lambda}(t)$ is reducible then there exist $x$ and $y$ in $\mathbb{F}_q^*$ such that
\[
p_{\lambda}(t) =  t^2 - (\lambda^{-1}d + \lambda)t+ 1 = (t - x)(t - y) = t^2 - (x+y)t + xy.
\] 
Hence, $xy=1$ and $x+y= \lambda^{-1}d + \lambda$. Assume without loss of generality that $y=x^{-1}$. If there exist values of $\lambda$ such that $g_{\lambda}$ has eigenvalues $\{x,x^{-1}\}$, then they have to satisfy the following quadratic equation
\begin{equation}\label{ecu2}
\lambda^2 - (x + x^{-1})\lambda + d =0.
\end{equation}

\begin{itemize}
\item \textbf{Case 4 (a)}:

\hspace{0.5cm}If we assume $d=0$ then $\lambda=0$ is a solution of Equation (\ref{ecu2}), however, that solution is not admissible by the definition of $g_{\lambda}$. Hence,  we just consider the solution $\lambda = x + x^{-1}$ for every $x \in \mathbb{F}_q^*$. Moreover, note that $x$ and $x^{-1}$ generate the same value of $\lambda$. In fact, we can relate to each set $\{x, x^{-1}\}$ a unique value of $\lambda$.  

\hspace{0.5cm}Let $q \equiv 1 \pmod 4$  and  $k \in \mathbb{F}_q^*$ be an element of order $4$. Note that the set $\{k, k^{-1}\}$ does not generate any admissible value of $\lambda$. Thus, the number of values of  $\lambda$ such that $p_{\lambda}(t)$ is reducible is $(q-1)/2$. Therefore,
\[
N_{(0, \infty),(1,0)} = \frac{1}{2} \left( q-1 - \frac{q-1}{2} \right) = \frac{q-1}{4}.
\]

\hspace{0.5cm}On the other hand, if $q \equiv 3 \pmod 4$ then $\mathbb{F}_q^*$ does not have an element of order 4. This implies that every set $\{x, x^{-1}\} \subset \mathbb{F}_q^*$ generates an admissible value of $\lambda$. Thus, the number of values for $\lambda$ such that $p_{\lambda}(t)$ is reducible is $(q+1)/2$ and $N_{(0,\infty),(1,0)}=(q-3)/4$.

\item \textbf{Case 4 (b)}:

\hspace{0.5cm}The number of solutions of Equation (\ref{ecu2}) in $\mathbb{F}_q$ is given by $1 + \phi((x + x^{-1})^2 -4d)$. In this case, $x$ and $x^{-1}$ leads to the same value of $\lambda$. Thus,  the number of values of  $\lambda \in \mathbb{F}_q^*$ such that $p_{\lambda}(t)$ is reducible is
\[
2(1+\phi(1-d))  + \frac{1}{2} \sum_{\substack{x \in \mathbb{F}_q^* \\ x \neq 1,-1 }} (1 + \phi((x + x^{-1} )^2 - 4d)).
\]
Therefore, for $d\neq 0,1,\infty$,
\begin{eqnarray*}
N_{(0, \infty),(1,d)} & = & \frac{1}{2} \left\{ (q-1) - \left[ 2(1+\phi(1-d))  + \frac{1}{2} \sum_{\substack{x \in \mathbb{F}_q^* \\ x \neq 1,-1 }} (1 + \phi((x + x^{-1} )^2 - 4d)) \right] \right\} \\
	& = & \frac{q-3}{4} - \frac{\phi(1-d)}{2}  - \frac{1}{4} \sum_{\substack{x \in \mathbb{F}_q^* \\ x \neq 1,-1 }} \phi((x + x^{-1} )^2 - 4d) 
\end{eqnarray*}
which gives the desired formula for $N_{(0, \infty),(1,d)}$.
\end{itemize}

\end{itemize}
\end{proof}

\begin{corollary}\label{usefulcor}
Let $d \in \mathbb{F}_q$, $d\neq 0,1$. The number of derangements of $PSL(2,q)$ sending $0$ to $\infty$ and $1$ to $d$ can be expressed in terms of the Legendre sum with respect to $\phi$. Specifically,
\begin{equation}\label{ecu10}
N_{(0,\infty),(1,d)} = \frac{q-1}{4} - \frac{\phi(1-d)}{2}-\frac{q}{4} P_{\phi}(2d-1).
\end{equation}
\end{corollary}

\begin{proof}
To prove this corollary, we compute
\begin{eqnarray*}
 \sum_{x \in \mathbb{F}_q^*} \phi((x + x^{-1})^2- 4d )& = & \sum_{x \in \mathbb{F}_q^*}  \phi(x^2)\phi((x + x^{-1})^2- 4d ) \\
  & = &  \sum_{x \in \mathbb{F}_q^*} \phi( x^4 - 2(2d-1)x^2 + 1).\end{eqnarray*}
Next we replace $x^2$ by $y$. If $y\in\F_q^*$ is not a square, then $1+\phi(y)=0$; on the other hand, if $y\in \mathbb{F}_q^*$  is a square, then  $x^2=y$ has $1+\phi(y)=2$ solutions. It follows that
  
\begin{eqnarray*}
 \sum_{x \in \mathbb{F}_q^*} \phi((x + x^{-1})^2- 4d )                                                                                    & = &  \sum_{y \in \mathbb{F}_q^*} (1 + \phi(y)) \phi(y^2 - 2(2d-1)y + 1)\\
                                                                                    & = &  \sum_{ y \in \mathbb{F}_q^*} \phi(y^2 - 2(2d-1)y + 1) + \sum_{y \in \mathbb{F}_q^*} \phi(y) \phi(y^2 - 2(2d-1)y + 1)\\
                                                                                    & = & -1  + \sum_{y \in \mathbb{F}_q} \phi(y^2 - 2(2d-1)y + 1)  + q P_{\phi} (2d-1).
\end{eqnarray*}

Applying Theorem 5.48 from \cite{FiniteFields} it follows that,
\[
	\sum_{y \in \mathbb{F}_{ q }} \phi(y^2 - 2(2d-1)y + 1)  = -1.
\]
Thus, the above computations imply that
\begin{equation}\label{ecu_extra1}
	 \sum_{x \in \mathbb{F}_q^*} \phi((x + x^{-1})^2- 4d ) = -2  + q P_{\phi} (2d-1).
\end{equation}
Now, Corollary \ref{usefulcor} follows from part 4(b) of Lemma \ref{lemma1}  and Equation (\ref{ecu_extra1}).
\end{proof}

%%%%%%%%%%%%%%%%%%%%%%%%%%%%%%%%%%
\subsection{A permutation $PGL(2,q)$-module}
%%%%%%%%%%%%%%%%%%%%%%%%%%%%%%%%%%

In this section we define a $PGL(2, q)$-module $V$ and a $PGL(2, q)$-module homomorphism $T_N$ from $V$ to $V$. We use the subscript $N$ to emphasize that $N$ is the matrix associated with $T_N$ with respect to a certain basis of $V$.

Recall that we denote by $\Omega$ the set of ordered pairs of distinct projective points in $PG(1,q)$. Let $V$ be the $\mathbb{C}$-vector space spanned by the vectors $\{e_{\omega}\}_{\omega \in \Omega}$. The dimension of $V$ is $q(q+1)$.

We define a right action of $PGL(2,q)$ on the basis $\{e_{\omega}\}$ of $V$. Specifically, if $\omega=(a,b)$ then 
$$e_{\omega} \cdot g = e_{\omega^g} = e_{(a^g,b^g)}$$
for any $g \in PGL(2,q)$.  Thus, $V$ is a right permutation $PGL(2, q)$-module. The next lemma shows that $V$ has a very simple decomposition into irreducible modules; apart from $V_{\lambda_{-1}}$ and $V_{\psi_{1}}$ each irreducible module of $PGL(2,q)$ appears exactly once.

Let $(\chi,\psi)$ denote the inner product of the characters $\chi$ and $\psi$ of $PGL(2,q)$ (see \cite[Section 2.3]{Serre}).

\begin{lemma}\label{lemma2}
Let $V_{\chi}$ denote an irreducible module of $PGL(2,q)$ with character $\chi$. Then the decomposition of V into irreducible constituents is given by,
\[
V \cong V_{\lambda_1} \oplus 2V_{\psi_1} \oplus V_{\psi_{-1}} \oplus \bigoplus_{\beta \in B} V_{\eta_{\beta}} \oplus \bigoplus_{\gamma \in  \Gamma } V_{\nu_{\gamma}}. 
\] 
\end{lemma}

\begin{proof}
Let $\pi$ be the character afforded by the $PGL(2,q)$-module $V$. By definition we have 
\[
\pi(g) := |\{ \omega \in \Omega :  \omega^g = \omega\}|
\]
hence the character $\pi$ has an easy description given by the following table
\begin{center}
\begin{tabular}{ c | c c c c c } 
             & 1          & $u$ & $d_x$  & $v_r$ \\ \hline
  $\pi$   & $q(q+1)$ & 0     &  2         & 0 \\  
\end{tabular}.
\end{center}

Now let $V_{\chi}$ be an irreducible representation of $PGL(2,q)$ and $\chi$ its irreducible character. It is known (\cite[ Chapter 2, Theorem 4]{Serre}) that the multiplicity of $V_{\chi}$ in $V$ is equal to the character inner product $(\pi, \chi)$.  
Thus, the lemma follows  by direct calculation using the character table of $PGL(2,q)$.
\end{proof}

For $a,b\in PG(1,q)$ with $a \neq b$, consider the following vectors in V,
\begin{eqnarray}
l_{a,b} & {:=} & \sum_{\substack{p \in PG(1,q) \\ p \neq a,b}} (e_{(a,p)} - e_{(b,p)} ) + e_{(a,b)} - e_{(b,a)} \label{v_ecu4},\\
r_{a,b} & {:=} & \sum_{\substack{p \in PG(1,q) \\ p \neq a,b}} (e_{(p,a)} - e_{(p,b)} ) + e_{(b,a)} - e_{(a,b)} \label{v_ecu5}.
\end{eqnarray}
We use these vectors to define the following vector subspaces of $V$,
\[
V_1 := \mbox{span}_{\mathbb{C}}\{ l_{a,b} : a,b \in PG(1,q), a \neq b \} \quad \mbox{and} \quad V_2 := \mbox{span}_{\mathbb{C}}\{ r_{a,b} : a,b \in PG(1,q), a \neq b \}. 
\]
In fact, the next lemma shows that $V_1$ and $V_2$ are $PGL(2,q)$-submodules of $V$.

\begin{lemma}\label{lemma4}
The vector subspaces $V_1$ and $V_2$ satisfy the following properties:
\begin{enumerate}
\item $\dim_{\mathbb{C}}(V_1) = \dim_{\mathbb{C}}(V_2) = q$,
\item $V_1 \cap  V_2 = \{ 0 \}$,
\item $V_1$ and $V_2$ are $PGL(2,q)$-submodules of $V$,
\item $V_1 \cong V_2 $ as $PGL(2,q)$-modules.
\end{enumerate}
\end{lemma}

\begin{proof}
Note that the vectors defined in Equations (\ref{v_ecu4}) and (\ref{v_ecu5}) satisfy the following relations,
\[
l_{a,b} - l_{a,c} = l_{c,b} \quad \mbox{ and } r_{a,b} - r_{a,c} = r_{c,b}
\]
for all $a,b,c \in PG(1,q)$ with $a \neq b \neq c$. Hence, fixing $a \in PG(1,q)$ we see that $\{ l_{a,b} : b \in PG(1,q), b \neq a \}$ and $\{ r_{a,b}: b \in PG(1,q), b \neq a\}$ are basis for $V_1$ and $V_2$, respectively. 

To prove the conclusion in part (2) we proceed by contradiction. Assume there exists $v \in V_1 \cap V_2$ with $v \neq 0$. Hence we can write
\begin{equation}\label{ecu6}
v = \sum_{\substack{p \in PG(1,q) \\  p \neq a}} \alpha_p l_{a,p} = \sum_{\substack{p \in PG(1,q)\\ p \neq a}} \beta_p r_{a,p}
\end{equation}
where not all $\alpha_p$ and $\beta_p$ are equal to zero.

For a fixed $b \in PG(1,q)$, the vector $l_{a,b}$ is the only one in the set $\{l_{a,p}\}_{p \in PG(1,q), p \neq a}$ that contains $e_{(b,a)}$. On the other hand, every vector of the form $r_{a,p}$ with $p \neq a$ contains $e_{(b,a)}$. Therefore, using Equation (\ref{ecu6}) we get
\[
\alpha_b = \sum_{\substack{p \in PG(1,q)\\ p \neq a}} \beta_p,
\]
which implies that the values of the coefficients $\alpha_p$ in Equation (\ref{ecu6}) are all the same. Analogously, we can show that the values $\beta_p$ in Equation (\ref{ecu6}) are the same. Thus, we can rewrite Equation (\ref{ecu6}) as follows,
\[
\alpha \sum_{\substack{p \in PG(1,q) \\  p  \neq a}} l_{a,p} = \beta \sum_{\substack{p \in PG(1,q)\\ p \neq a}}  r_{a,p}
\]
where $\alpha = \sum_{p\neq a} \beta_p$ and $\beta= \sum_{p\neq a} \alpha_p$. This implies that $\alpha = q \beta = q^2 \alpha$, a contradiction, because $q$ is not equal to one.

To prove part (3) it is enough to note that $l_{a,b} \cdot g= l_{a^g,b^g}$ and $r_{a,b} \cdot g = r_{a^g, b^g}$ for all $a,b \in PG(1,q)$ with $a \neq b$.  For part (4) consider the function $\theta$ from $V_1$ to $V_2$ defined by $\theta(l_{a,b})=r_{a,b}$ for all $a,b \in PG(1,q)$ with $a \neq b$; we extend the definition of $\theta$ to all elements of $V_1$ linearly. Now, from the definition of $\theta$ we see that clearly 
\[
\theta( l_{(a,b)} \cdot g) =  \theta(l_{(a,b)}) \cdot g
\]
for all $g \in PGL(2,q)$ and $(a,b) \in \Omega$. Therefore, $\theta$ is a $PGL(2,q)$-module isomorphism. This completes the proof of part (4). 
\end{proof}

\begin{lemma}\label{3_lemma1}
The submodules $V_1$ and $V_2$ are isomorphic to $V_{\psi_1}$.
\end{lemma}

\begin{proof}
This result follows directly from Lemmas \ref{lemma2} and \ref{lemma4}. If we consider the decomposition of $V$ into irreducible constituents, we note that each irreducible representation appears only once, except for $V_{\psi_1}$ . Therefore, because $V_1$ is isomorphic to $V_2$, we must have $V_{\psi_1} \cong V_1 \cong V_2$.
\end{proof}

We now define a linear transformation $T_N$  from $V$ to $V$. We first define $T_N$ on the basis $\{e_{\omega}\}_{\omega \in \Omega}$ of $V$ by
\[
T_N(e_{(a,b)}) := \sum_{\omega \in \Omega} N_{\omega, (a,b)} e_{\omega}
\]
for any $(a,b) \in \Omega$, and then extend the definition of $T_N$ to all elements of $V$ linearly. It follows from the definition of $T_N$ that $N$ is the matrix associated with $T_N$ with respect to the basis $\{e_{\omega}\}_{\omega \in \Omega}$ of $V$. Therefore, the dimension of the image of $T_N$ is equal to the rank of the derangement matrix $M$ of $PSL(2,q)$ acting on $PG(1,q)$.

\begin{lemma}\label{lemma3}
 The linear transformation $T_N$ defined above  is a $PGL(2,q)$-module homomorphism from $V$ to $V$.
\end{lemma}

\begin{proof}
To prove the lemma we have to show that the linear transformation $T_N$ respects the action of $PGL(2,q)$ on $V$; that is, for each $g \in PGL(2,q)$ and each $(a,b) \in \Omega$, 
\begin{equation}\label{ecu3}
T_N( e_{(a,b)} \cdot g) =  T_N(e_{(a,b)}) \cdot g.
\end{equation}

First, consider the left hand side of Equation (\ref{ecu3}). From the definition of $T_N$ it follows that
\[
T_N( e_{(a,b)} \cdot g) = T_N(e_{(a^g,b^g)}) = \sum_{\omega \in  \Omega} N_{\omega, (a^g,b^g)} e_{\omega}.
\] 

Now, note that the right hand side of  Equation (\ref{ecu3})  can be written as
\[ 
 T_N(e_{(a,b)}) \cdot g =  \sum_{\omega \in \Omega} N_{\omega, (a,b)} e_{\omega^g} =  \sum_{\omega^{g^{-1}} \in \Omega} N_{\omega^{g^{-1}}, (a,b)} e_{\omega}. 
\]

Furthermore, recall that $N_{(a,b),(c,d)} = N_{(a^g,b^g),(c^g,d^g)}$ for all $g \in PGL(2,q)$. Therefore,
\[
\sum_{\omega^{g^{-1}} \in \Omega} N_{\omega^{g^{-1}}, (a,b)} e_{\omega} = \sum_{\omega^{g^{-1}} \in  \Omega} N_{\omega, (a^g,b^g)} e_{\omega} =\sum_{\omega \in  \Omega} N_{\omega, (a^g,b^g)} e_{\omega}
\] 
which implies that Equation (\ref{ecu3}) holds.  This completes the proof of the lemma.
\end{proof}

%%%%%%%%%%%%%%%%%%%%%%%%%%%%%%%%%%
\subsection{The image of $T_N$}\label{psl_im_TN}
%%%%%%%%%%%%%%%%%%%%%%%%%%%%%%%%%%

Recall that the rank of the derangement matrix $M$ of $PSL(2,q)$ acting on $PG(1,q)$ is equal to the dimension of the image of $T_N$. Since $T_N$ is a $PGL(2, q)$-module homomorphism (Lemma \ref{lemma3}) we can use some tools from representation theory to compute the dimension of the image of $T_N$. We start by observing that the submodules $V_1$ and $V_2$ are in the kernel of $T_N$.

\begin{lemma}\label{3_lemma2}
The subspaces $V_1$ and $V_2$ lie in the kernel of $T_N$.
\end{lemma}

\begin{proof}

First, recall that the derangement matrix $M$ is a $q(q-1)^2/4$ by $(q+1)q$ matrix whose rows are indexed by the derangements of $PSL(2,q)$ and whose columns are indexed by elements of $\Omega$. For any derangement $g \in PSL(2,q)$ and $(a,b) \in \Omega$ we have 
\[
M(g, (a,b)) {:=} \left\lbrace \begin{array}{cl}
1, & \mbox{ if }a^g = b,\\
0, & \mbox{ otherwise.}
\end{array} \right.
\]
Furthermore, also by definition, we have $N=M^{\top}M$. Thus, the lemma follows from the following observation
\[
Ml_{a,b}=0 \quad \mbox{and} \quad Mr_{a,b}=0 \quad \mbox{for all }a,b \in PG(1,q),\mbox{ with }a\neq b,
\]
and the fact that for a fixed $a \in PG(1,q)$ the sets $\{l_{a,b}: b \in PG(1,q), b\neq a \}$ and $\{r_{a,b}: b \in PG(1,q), b\neq a \}$ are basis of $V_1$ and $V_2$, respectively.

\end{proof}

From Lemma \ref{3_lemma1} and \ref{3_lemma2}, we conclude that the restriction of $T_N$ to $2V_{\psi_1}$ is the zero map. It follows that the dimension of the image of $T_N$ is at most $q(q-1)$. Now, we consider the restriction of $T_N$ onto  the other irreducible constituents of $V$. To do that we apply Schur's lemma.

Let $\chi$ be the irreducible character corresponding to an irreducible representation of $PGL(2,q)$ appearing as a constituent of $V$. Schur's lemma implies that,
\[
T_N(V_{\chi}) \cong V_{\chi} \quad \mbox{ or } \quad T_N(V_{\chi})= \{0\}.
\]
Thus, either the dimension of the restriction of $T_N$ to $V_{\chi}$ is zero or is equal to the dimension of $V_{\chi}$. Hence, 
to study the image of $V_{\chi}$ under $T_N$ for any $\chi \in \{ \lambda_1, \psi_{-1}, \{\eta_{\beta} \}_{\beta \in B}, \{ \nu_{\gamma}\}_{\gamma \in \Gamma} \}$ we proceed in the following way:
\begin{enumerate}
\item Consider the vector $e_{(0,\infty)} \in V$.
\item Project $e_{(0,\infty)}$ onto $V_{\chi}$ using the following scalar multiple of a central primitive idempotent
		\[  E_{\chi} := \sum_{g \in PGL(2,q)} \chi(g^{-1}) g. \]
		
		Therefore, the projection of $e_{(0,\infty)}$ onto $V_{\chi}$ is equal to
		\[  E_{\chi}(e_{(0,\infty)}) = \sum_{g \in PGL(2,q) } \chi(g^{-1}) e_{(0^g,\infty^g)}  = \sum_{(a,b) \in \Omega} \left[  \sum_{0^g=a,\infty^g=b} \chi(g^{-1}) \right] e_{(a,b)}. 
		\]
		where $g$ in the inner sum runs over all elements in $PGL(2,q)$ sending $0$ to $a$ and $\infty$ to $b$.
		
\item To prove that $T_N(V_{\chi}) \cong V_{\chi}$  it  is enough to show that the $(0,\infty)$ coordinate of $T_N(E_{\chi}(e_{(0,\infty)}))$ is not equal to zero. This is equivalent to  showing that the following character sum is not equal to zero:
\begin{equation}\label{ecu7}
T_{N,\chi} := T_N(E_{\chi}(e_{(0,\infty)}))_{(0,\infty)} = \sum_{(a,b) \in \Omega} \left[ \sum_{0^g=a,\infty^g=b} \chi(g^{-1}) \right] N_{(0,\infty),(a,b)},
\end{equation} 
where $g$ in the inner sum runs over all elements in $PGL(2,q)$ sending $0$ to $a$ and $\infty$ to $b$.
%for every irreducible character $\chi \in  \{ \lambda_1, \psi_{-1}, \{\eta_{\beta} \}_{\beta \in B}, \{ \nu_{\gamma}\}_{\gamma \in \Gamma} \}$.
\end{enumerate}

Therefore, we get the following lower bound on the rank of the derangement matrix $M$, 
\begin{equation}\label{INEQU}
	\sum_{\chi} \dim(V_{\chi}) \leq \mbox{rank}(M),
\end{equation}
where $\chi$ in the sum on the left hand side of (\ref{INEQU}) runs through $\{ \lambda_1, \psi_{-1}, \{\eta_{\beta} \}_{\beta \in B}, \{ \nu_{\gamma}\}_{\gamma \in \Gamma} \}$ such that $T_{N,\chi} \neq 0$. In particular,  if $T_{N,\chi}$ is not zero for all   $\chi \in  \{ \lambda_1, \psi_{-1}, \{\eta_{\beta} \}_{\beta \in B}, \{ \nu_{\gamma}\}_{\gamma \in \Gamma} \}$ then the rank of the derangement matrix $M$ is equal to $q(q-1)$. We conclude that to prove Theorem \ref{psl_teo2},  it is enough to show that the values of the character sums $T_{N,\chi}$ with $\chi \in  \{ \lambda_1, \psi_{-1}, \{\eta_{\beta} \}_{\beta \in B}, \{ \nu_{\gamma}\}_{\gamma \in \Gamma} \}$ are not equal to zero.  This will be our objective in the next two sections.

%%%%%%%%%%%%%%%%%%%%%%%%%%%
\section{The character sums $\displaystyle \sum_{0^g = \infty, \infty^g=0} \chi(g^{-1})$ and $\displaystyle \sum_{0^g = \infty, 1^g=d} \chi(g^{-1})$}
%%%%%%%%%%%%%%%%%%%%%%%%%%%

The sums $T_{N,\chi}$ are character  sums over $PGL(2,q)$. In general, it is not easy to get tight bounds on the values of characters sums over non-abelian groups. Fortunately, the close relationship between the irreducible characters of $PGL(2,q)$ and the multiplicative characters of $\F_q$ and $\F_{q^2}$ allows us to conclude in Section 5 that the expressions $T_{N,\chi}$ are not equal to zero. In this section, we show that we can express the sums $T_{N,\chi}$ in terms of characters sums over finite fields for every $\chi \in  \{ \lambda_1, \psi_{-1}, \{\eta_{\beta} \}_{\beta \in B}, \{ \nu_{\gamma}\}_{\gamma \in \Gamma} \}$.

First, we consider $T_{N,\chi}$ when $\chi= \lambda_1$.  In this case, we know that $\lambda_1(g)=1$ for any $g \in PGL(2,q)$. Moreover, there are precisely $q-1$ elements of $PGL(2,q)$ sending  $0$ to $a$ and $\infty$ to $b$ for any $a,b \in PG(1,q)$. Therefore, we can compute (\ref{ecu7}) explicitly for $\chi = \lambda_1$:
\[
T_{N,\lambda_1} = (q-1) \sum_{(a,b) \in \Omega} N_{(0,\infty)(a,b)}= (q-1)(q+1)\frac{(q-1)^2}{4},
\]
where we have used Lemma \ref{lemma1} to obtain the last equality. Thus, from the analysis given in Section \ref{psl_im_TN} we conclude that $T_N(V_{\lambda_1}) \cong V_{\lambda_1}$. 

The other irreducible characters of $PGL(2,q)$ are not so easy to handle. The next lemma gives an expression for  $T_{N,\chi}$ with $\chi \in  \{ \psi_{-1}, \{\eta_{\beta} \}_{\beta \in B}, \{ \nu_{\gamma}\}_{\gamma \in \Gamma} \}$ which will be helpful to write Equation (\ref{ecu7}) in terms of character sums over finite fields. 

\begin{lemma}\label{lemma7}
Let  $\chi$ be any irreducible character of $PGL(2,q)$ from the set $ \{ \psi_{-1}, \{\eta_{\beta} \}_{\beta \in B}, \{ \nu_{\gamma}\}_{\gamma \in \Gamma} \}$. Let $h$ be the unique element of $PGL(2,q)$ sending $0$ to $0$, $1$ to $\infty$, and $\infty$ to $1$. If $q \equiv 1 \mbox{ mod }4$ then
\[
T_{N,\chi} = \frac{(q-1)^3}{4} - \frac{q-1}{2} \sum_{0^g = \infty, \infty^g=0} \chi(g^{-1}) + (q-1)  \sum_{\substack{b \in \mathbb{F}_q^* \\ b \neq 1}}  \left[ \sum_{0^g=\infty,  1^g =b^h}  \chi(g^{-1}) \right] N_{(0,\infty),(1,b)},
\]
and if $q \equiv 3 \mbox{ mod }4$ then
\[
T_{N,\chi} = \frac{(q-1)^3}{4} + \sum_{0^g = \infty, \infty^g=0} \chi(g^{-1}) + (q-1)  \sum_{\substack{b \in \mathbb{F}_q^* \\ b \neq 1}}  \left[ \sum_{0^g=\infty,  1^g =b^h}  \chi(g^{-1}) \right] N_{(0,\infty),(1,b)}.
\]
\end{lemma}

\begin{proof}
We start by presenting some results on character sums over $PGL(2,q)$ that we will need.

We denote by $PGL(2,q)_{0,\infty}$ the subgroup of $PGL(2,q)$ fixing $0$ and $\infty$. Analogously,  $PGL(2,q)_0$ denotes the subgroup of $PGL(2,q)$ fixing $0$. Applying the Frobenius Reciprocity Theorem  \cite[Chapter 7, Theorem 13]{Serre}, we have: 
\[
( \mbox{Res}(\chi) , 1 )_{PGL(2,q)_{0,\infty}}  = ( \chi , \pi )_{PGL(2,q)} \quad \mbox{and} \quad   ( \mbox{Res}(\chi) , 1 )_{PGL(2,q)_{0}}  = ( \chi , \lambda_1 + \psi_1 )_{PGL(2,q)}
\]
where $\pi$ is the permutation character defined in the proof of Lemma \ref{lemma2} and $1$ is the trivial character of the groups $PGL(2,q)_{0,\infty}$ and $PGL(2,q)_{0}$, respectively. Using these equalities and the decomposition of $\pi$ in terms of irreducible characters (which was given in Lemma \ref{lemma2}), we evaluate the following character sums:
\begin{equation*}
	\sum_{0^g = 0, \infty^g=\infty} \chi(g^{-1})   =  (q-1) ( \mbox{Res}(\chi) , 1 )_{PGL(2,q)_{0,\infty}}  = (q-1) ( \chi , \pi )_{PGL(2,q)} = q-1,
\end{equation*}
and
\begin{equation*}
	 \sum_{0^g = 0} \chi(g^{-1})   =  q(q-1)  ( \mbox{Res}(\chi) , 1 )_{PGL(2,q)_{0}}  = q(q-1) ( \chi , \lambda_1 + \psi_1 )_{PGL(2,q)} = 0. 
\end{equation*}

Note that $\chi(kgk^{-1})= \chi(g)$ for any $k \in PGL(2,q)$ since $\chi$ is a character, hence a class function. This fact implies many relations between character sums over $PGL(2,q)$. In particular, 
\begin{equation}\label{ecu_extra2}
\sum_{a^g =b} \chi(g^{-1}) = \sum_{(a^k)^g =(b^k)^g} \chi(g^{-1}), 
\end{equation}
and
\begin{equation}\label{ecu_extra3}
\sum_{a^g = b, c^g=d} \chi(g^{-1}) = \sum_{(a^k)^g = b^k, (c^k)^g=d^k} \chi(g^{-1}).
\end{equation}
We claim that $\sum_{0^g =\infty} \chi(g^{-1}) =0$. To prove this claim, recall that $\chi$ is a non-trivial character of $PGL(2,q)$. Therefore,
\[
	0= \sum_{g \in PGL(2,q)} \chi(g^{-1}) = \sum_{0^g = 0} \chi(g^{-1}) + \sum_{\substack{a \in PG(1,q)\\a \neq 0}} \sum_{0^g = a} \chi(g^{-1}). 
\]
Since $\sum_{0^g = 0} \chi(g^{-1})=0$, we conclude that
\[
	0 = \sum_{\substack{a \in PG(1,q)\\a \neq 0}} \sum_{0^g = a} \chi(g^{-1})= q \sum_{0^g = \infty} \chi(g^{-1}),
\]
where Equation (\ref{ecu_extra2}) is  used to obtain the last equality.

Moreover, it follows from the above equations and the $2$-transitivity of the action of $PGL(2,q)$ on $PG(1,q)$ that
\[
\sum_{\infty^g =\infty} \chi(g^{-1}) =0 \quad \mbox{and} \quad \sum_{\infty^g = 0} \chi(g^{-1})  = 0.
\]

Now, we are ready to prove Lemma \ref{lemma7}. From  Equation (\ref{ecu7}) and Lemma \ref{lemma1} we get,
\begin{eqnarray*}
T_{N,\chi}  & = & \frac{(q-1)^2}{4} \sum_{0^g = 0, \infty^{g}=\infty} \chi(g^{-1}) + \left[ \sum_{0^g = \infty, \infty^g=0} \chi(g^{-1}) \right] N_{(0,\infty),(\infty,0)} \\
 & & + \sum_{b \in \mathbb{F}_q^*} \left[ \sum_{0^g = \infty, \infty^g=b} \chi(g^{-1}) \right] N_{(0, \infty),(\infty,b)} + \sum_{a \in \mathbb{F}_q^*} \left[ \sum_{0^g = a, \infty^g=0} \chi(g^{-1}) \right] N_{(0, \infty),(a,0)} \\
 & & + \sum_{\substack{a,b \in \mathbb{F}_q^* \\ a \neq b}}  \left[ \sum_{0^g = a, \infty^g=b} \chi(g^{-1}) \right] N_{(0, \infty),(a,b)}. 
\end{eqnarray*}
First, assume that $q \equiv 1 \pmod 4$.  From Lemma \ref{lemma1} it follows that  $$N_{(0,\infty),(\infty,b)} =N_{(0,\infty),(a,0)} = (q-1)/4$$ for all $a,b \in \mathbb{F}_q^*$, and $$N_{(0, \infty),(\infty, 0)} = 0.$$ Hence, using the above analysis we can write,
\begin{eqnarray*}
 \sum_{b \in \mathbb{F}_q^*} \left[ \sum_{0^g = \infty, \infty^g=b} \chi(g^{-1}) \right] N_{(0, \infty),(\infty,b)}  & = &   \frac{q-1}{4} \sum_{b \in \mathbb{F}_q^*} \left[ \sum_{0^g = \infty, \infty^g=b} \chi(g^{-1}) \right]\\
& = &  \frac{q-1}{4} \left[\sum_{0^g =\infty} \chi(g^{-1})  - \sum_{0^g = \infty, \infty^g=0} \chi(g^{-1}) \right]\\
& = &   - \frac{(q-1)}{4}\sum_{0^g = \infty, \infty^g=0} \chi(g^{-1}),
\end{eqnarray*}
and using the same ideas we get
\[
\sum_{a \in \mathbb{F}_q^*} \left[ \sum_{0^g = a, \infty^g=0} \chi(g^{-1}) \right] N_{(0, \infty),(a,0)} = - \frac{(q-1)}{4}\sum_{0^g = \infty, \infty^g=0} \chi(g^{-1}).
\]

Let $a,b \in \mathbb{F}_q^*$ with $a \neq b$. Using the $3$-transitivity of the action of $PGL(2,q)$ on $PG(1,q)$ and (\ref{ecu1}) we conclude that $N_{(0,\infty),(a,b)}= N_{(0,\infty)(1,b^k)}$ where $k \in PGL(2,q)$ is the unique element sending $0$ to $0$, $\infty$ to $\infty$ and $a$ to $1$. Moreover, applying Equation (\ref{ecu_extra3}) we obtain 
\[
\sum_{0^g = a, \infty^g=b} \chi(g^{-1}) = \sum_{0^g = 1, \infty^g=b^k} \chi(g^{-1}).
\]  
Putting all these facts together we conclude that
\begin{eqnarray*}
\sum_{\substack{a,b \in \mathbb{F}_q^* \\ a \neq b}}  \left[ \sum_{0^g = a, \infty^g=b} \chi(g^{-1}) \right] N_{(0, \infty),(a,b)}  & = & (q-1)  \sum_{\substack{b \in \mathbb{F}_q^* \\ b \neq 1}}  \left[ \sum_{0^g=1,  \infty^g =b}  \chi(g^{-1}) \right] N_{(0,\infty),(1,b)} \\
& = & (q-1)  \sum_{\substack{b \in \mathbb{F}_q^* \\ b \neq 1}}  \left[ \sum_{0^g=\infty,  1^g =b^h}  \chi(g^{-1}) \right] N_{(0,\infty),(1,b)}.
\end{eqnarray*}
Thus, Lemma \ref{lemma7} is proved for the case where $q \equiv 1 \mbox{ mod }4$. Similar computations work for the case when $q \equiv 3 \mbox{ mod }4$.

\end{proof}

It follows from Lemma \ref{lemma7} that we can write $T_{N, \chi}$ in terms of the character sums
\[
 \sum_{0^g = \infty, \infty^g=0} \chi(g^{-1}) \quad \mbox{ and } \quad \sum_{0^g = \infty, 1^g=d} \chi(g^{-1}).
\]
The next four lemmas show that these character sums can be written in terms of character sums over finite fields for all $\chi \in \{ \psi_{-1}, \{\eta_{\beta} \}_{\beta \in B}, \{ \nu_{\gamma}\}_{\gamma \in \Gamma} \}$.

\begin{lemma}\label{lemma8}
Let $i$ be an element of  $\mathbb{F}_{q^2}^* \setminus \mathbb{F}_q^*$ such that $i^2\in \mathbb{F}_q^*$. Then,
\begin{eqnarray*}
 \sum_{0^g = \infty, \infty^g=0} \psi_{-1}(g^{-1})  & = &  \phi(-1) (q-1), \\
 \sum_{0^g = \infty, \infty^g=0} \nu_{\gamma}(g^{-1})  & = & \gamma(-1)(q-1) \quad \mbox{ for all } \gamma \in \Gamma,\\
 \sum_{0^g = \infty, \infty^g=0} \eta_{\beta}(g^{-1})  & = &  -\beta(i)(q-1) \quad \mbox{ for all } \beta \in B.
\end{eqnarray*}  
\end{lemma}

\begin{proof} 
The elements in $PGL(2,q)$ sending  $0$ to $\infty$ and $\infty$ to $0$ are of the form,
\[
g_{\lambda} :=\left( \begin{array}{cc}
0 & \lambda\\
1 & 0
\end{array}\right) \quad \mbox{ with }\lambda \in \mathbb{F}_q^*.
\]
Note that the characteristic polynomial of $g_{\lambda}$ is $p_{\lambda}(t) := t^2 - \lambda$. 

To evaluate the character sums in this lemma we need to know to which conjugacy classes the elements $g_{\lambda}$ belong.  

First, recall that the eigenvalues of $g_{\lambda}$ are defined up to multiplication by an element of $\mathbb{F}_q^*$. Now, if $\lambda$ is a square in $\mathbb{F}_q^*$ then $p_{\lambda}(t)$ is reducible and $g_{\lambda}$ has eigenvalues $\pm \sqrt{\lambda} \in \mathbb{F}_q^*$. This implies  that $g_{\lambda}$ lies in the conjugacy class $d_{-1}$ whenever $\lambda$ is a square. On the other hand, if $\lambda$ is not a square the roots of $p_{\lambda}(t)$ lie on $\mathbb{F}_{q^2}^*$ and they correspond to elements of order $2$ in $\mathbb{F}_{q^2}^*/\mathbb{F}_{q}^*$. Therefore, whenever $\lambda$ is not a square we see that $g_{\lambda}$ lies on the conjugacy class $v_i$. 

Since there are equal number of squares and nonsquares in $\F_q^*$, the lemma follows from the character table of $PGL(2,q)$. 
\end{proof}

\begin{lemma}\label{lemma9}
For every $\gamma \in \Gamma$ and $d \in \mathbb{F}_q^*\setminus \{1\}$ we have
\[
\sum_{0^g = \infty, 1^g=d} \nu_{\gamma}(g^{-1}) = q P_{\gamma}(2d-1) .
\]
\end{lemma}

\begin{proof} 
The elements in $PGL(2,q)$ sending $0$ to $\infty$ and $1$ to $d$ are of the form,
\[
g_{\lambda} := \left( \begin{array}{cc}
0 & \alpha \lambda\\
\alpha &\alpha(d - \lambda)
\end{array}\right) \quad \mbox{ with }\lambda, \alpha \in \mathbb{F}_q^*. 
\]

To evaluate the sum in this lemma we need to know to which conjugacy classes these elements belongs. However, we need to do this just for those elements which are not derangements because $\nu_{\gamma}(g)=0$ if $g$ is a derangement.

Note that different values of $\alpha$ correspond to the same element $g_{\lambda}$ in $PGL(2,q)$. Indeed, as was remarked earlier the eigenvalues of $g_{\lambda}$ are defined up to scalar multiplication. 

The characteristic polynomial of $g_{\lambda}$ is $p_{\lambda}(t) := t^2 - \alpha(d - \lambda) t - \alpha^2 \lambda$ and its eigenvalues are,
\[
 \alpha \left( \frac{(d-\lambda) \pm \sqrt{(d - \lambda)^2 + 4\lambda})}{2} \right).
\]
Thus, if $\sqrt{(d - \lambda)^2 + 4\lambda} \in \mathbb{F}_q^*$ then there exists $\alpha \in \mathbb{F}_q^*$ such that the eigenvalues of $g_{\lambda}$ are $\{1,x\}$ for some $x \in \mathbb{F}_q^*$. This implies that $g_{\lambda}$ is contained in the same conjugacy class as $d_x$ (see Section \ref{psl_ct}).  Here, we assume that $d_x$ with $x=1$ corresponds to the element $u \in PGL(2,q)$ defined in Section \ref{psl_ct}.

For a fixed $d \in \mathbb{F}_q^*\setminus \{1\}$ and $x \in \mathbb{F}_q^*$ we want to know for how many $\lambda \in \mathbb{F}_q^*$  there exists some $\alpha$ such that $g_{\lambda}$ has eigenvalues $\{1,x\}$. From the above analysis it is clear that  $d,x, \alpha$ and $\lambda$ must satisfy the equation below: 
\[
p_{\lambda}(t)= t^2 - \alpha(d - \lambda) t - \alpha^2 \lambda = (t-x)(t-1)=t^2 - (x+1)t +x.
\]
This implies that $\alpha$ satisfies the following quadratic equation,
\begin{equation*}
d\alpha^2 - (x + 1) \alpha  + x = 0. 
\end{equation*}

Therefore, given $x \in \mathbb{F}_q^*$ and $d \in \mathbb{F}_q^*\setminus \{1\}$, the number of values of  $\lambda \in \mathbb{F}_q^*$  such that $g_{\lambda}$ is  conjugate to $d_x$ is equal to 
\[
1 + \phi((x+1)^2 - 4xd)  \mbox{ if }x \neq -1  \quad \mbox{and} \quad
(1 + \phi((x+1)^2 - 4xd))/2  \mbox{ if }x = -1.
\]

Now using the above remarks and the character table of $PGL(2,q)$ we get
\begin{eqnarray}\label{ecu8}
\sum_{0^g = \infty, 1^g=d} \nu_{\gamma}(g)  & = & (1+ \phi(1-d))\gamma(1)+ \left( \frac{1 + \phi(d)}{2}\right) (2\gamma(-1)) \\
&    & + \frac{1}{2} \sum_{\substack{x \neq 1, -1 \\ x \in \mathbb{F}_q^*} } (1 + \phi((x+1)^2 - 4xd) )  (\gamma(x) + \gamma(x^{-1})) \nonumber
\end{eqnarray}
where the first two terms in the right hand side of Equation (\ref{ecu8}) corresponds to $x=1$ and $x=-1$. Furthermore, note that we have included a factor $\frac{1}{2}$ in front of the last expression in Equation (\ref{ecu8}). This occurs because  every element $g_{\lambda}$ having eigenvalues $\{1,x\}$ also has eigenvalues $\{1, x^{-1}\}$. Hence, given $d \in \mathbb{F}_q^*\setminus \{1\}$, the elements $x$ and $x^{-1}$ are related to the same values of $\lambda$. 
Simplifying the right hand side of Equation (\ref{ecu8}),
\begin{eqnarray*}
\sum_{0^g = \infty, 1^g=d} \nu_{\gamma}(g) & = & \sum_{x \in \mathbb{F}_q^*} \gamma(x) \phi(x^2-2(2d-1)x+1)\\
& = & q P_{\gamma}(2d-1).
\end{eqnarray*}
Finally, applying basic properties of characters and Lemma \ref{lemma21} we obtain
\[
\sum_{0^g = \infty, 1^g=d} \nu_{\gamma}(g^{-1}) = \overline{\sum_{0^g = \infty, 1^g=d} \nu_{\gamma}(g) } =q P_{\gamma^{-1}}(2d-1)  = q P_{\gamma}(2d-1). 
\]
The proof is now complete.
\end{proof}

\begin{lemma}\label{lemma10}
For every $\beta \in B$ and $d \in \mathbb{F}_q^*\setminus \{1\}$ we have,
\[
\sum_{0^g = \infty, 1^g=d} \eta_{\beta}(g^{-1}) = - q R_{\beta}(2d-1). 
\]
\end{lemma}

\begin{proof} 
Recall that all the elements in $PGL(2,q)$ sending $0$ to $\infty$ and $1$ to $d$ take the form,
\[
g_{\lambda} := \left( \begin{array}{cc}
0 & \alpha\lambda\\
\alpha & \alpha(d - \lambda)
\end{array}\right) \quad \mbox{ with }\lambda, \alpha \in \mathbb{F}_q^*. 
\]

To evaluate the sum in this lemma we have to know to which conjugacy classes these elements belong.  However, since $\eta_{\beta}(g)=0$ if $g$ has two fixed points,  we will pay attention to  derangements and the elements fixing one point only (see Section \ref{psl_ct}).

We know that if $r \in \mathbb{F}_{q^2}^* \setminus \mathbb{F}_{q}^*$ is an eigenvalue of $g_{\lambda}$ then $g_{\lambda}$ is a derangement with eigenvalues $\{r, r^q\}$ contained in the same conjugacy class as $v_r$. On the other hand, if $r \in \mathbb{F}_{q}^*$ is the only eigenvalue of $g_{\lambda}$ then this implies that $g_{\lambda}$ has exactly one fixed point and it is conjugated to $u$. In fact, when $r \in \F_q^*$ every element of the form $v_r$ is conjugated to $u$. 

Fix $r \in \mathbb{F}_{q^2}^*$. We want to know for how many values of $\lambda \in \mathbb{F}_q^*$ there exists $\alpha$ such that $g_{\lambda}$ has eigenvalues $\{r, r^q\}$. From the characteristic polynomial of $g_{\lambda}$ the following equation is obtained
\[
 t^2 - \alpha(d - \lambda) t - \alpha^2 \lambda =t^2 - (r + r^q)t +r^{q+1}, 
\]
which implies that $\alpha \in \mathbb{F}_q^*$ must satisfy the quadratic equation below
\begin{equation}\label{ecu9}
d\alpha^2 - (r + r^q) \alpha + r^{q+1}=0.
\end{equation}

Distinct  solutions of Equation (\ref{ecu9}) generate distinct values of $\lambda$ unless $r \in i \mathbb{F}_q$ where $i$ is an element of $\mathbb{F}_{q^2}^* \setminus \mathbb{F}_{q}^*$ such that $i^2 \in \mathbb{F}_{q}^*$. Hence, given $r \in \mathbb{F}_{q^2}^*$ and $d \in \mathbb{F}_{q}^*\setminus \{1\}$, the number of  $\lambda \in \mathbb{F}_q^*$  such that  $g_{\lambda}$ is conjugated to $v_r$ is equal to: 
\[
1 + \phi((r +r^q)^2 - 4dr^{q+1})  \mbox{ if } r \in \mathbb{F}_{q^2}^*\setminus i\mathbb{F}_{q}^* \quad \mbox{and} \quad
(1 + \phi((r+r^q)^2 - 4dr^{q+1}))/2  \mbox{ if }r \in i\mathbb{F}_{q}^*.
\]

Moreover, note that every element $g_{\lambda}$ having eigenvalues $\{r,r^q\}$ also has eigenvalues $\{ar, (ar)^q\}$ for any $a  \in \mathbb{F}_{q}^*$. Thus,  $r$ and $ar$ are related to the same values of $\lambda$ for every $a \in \F_q^*$. Therefore,
\begin{eqnarray*}
\sum_{0^g = \infty, 1^g=d} \eta_{\beta}(g^{-1})  & = &  \frac{1}{q-1} \sum_{r \in \mathbb{F}_q^*} (1 + \phi((r +r^q)^2 - 4dr^{q+1}))(-\beta(1))  \\
&  &  + \frac{1}{q-1} \sum_{r \in i\mathbb{F}_q^*}  \left(\frac{1 + \phi((r +r^q)^2 - 4dr^{q+1})}{2} \right) (-2\beta(i)) \\
&  &  + \frac{1}{2(q-1)} \sum_{r \in \mathbb{F}_{q^2}^*\setminus \{\mathbb{F}_{q}^* , i \mathbb{F}_q^*\} }  (1 + \phi((r +r^q)^2 - 4dr^{q+1})) (-\beta(r) - \beta(r^q)) \\
& = & \frac{1}{2(q-1)} \sum_{r \in \mathbb{F}_{q^2}^*} \phi((r +r^q)^2 - 4dr^{q+1}) (-2\beta(r) ) \\
& = &  -\frac{1}{q-1} \sum_{r \in \mathbb{F}_{q^2}^*} \phi((r +r^q)^2 - 4dr^{q+1}) \beta(r)
\end{eqnarray*}
Now, the lemma follows from Definition \ref{def1_1} and Lemma \ref{lemma21}.

\end{proof}

\begin{lemma}\label{lemma11}
For every $d \in \mathbb{F}_q^*\setminus \{1\}$ we have,
\[
\sum_{0^g = \infty, 1^g=d} \psi_{-1}(g) =  q P_{\phi}(2d-1) .
\]
\end{lemma}

\begin{proof}
From the character table of $PGL(2,q)$ it follows that
\begin{equation}\label{ecu17}
\psi_{-1}(g) =  \left\lbrace \begin{array}{ll}
  0, & \mbox{ if } g  \in u,\\
  1, & \mbox{ if }g \in d_x \mbox{ and } d_x \subset PSL(2,q),\\
  -1, & \mbox{ if }g \in d_x \mbox{ and } d_x \subset PGL(2,q) \setminus PSL(2,q),\\
  -1, & \mbox{ if }g \in v_r \mbox{ and } v_r \subset PSL(2,q),\\
    1, & \mbox{ if }g \in v_r \mbox{ and } v_r \subset PGL(2,q) \setminus PSL(2,q).\\  
\end{array}  \right.
\end{equation}

Thus, to evaluate the sum  $\sum_g \psi_{-1}(g)$  we need  to know: how many elements sending $0$ to $\infty$ and $1$ to $d$ belong to each of the five categories considered in (\ref{ecu17}). In fact, these counting problems follow from the proof of Case (4) of Lemma \ref{lemma1}.

For the sake of clarity, we recall some simple facts. There are $q-1$ elements in $PGL(2,q)$ sending $0$ to $\infty$ and $1$ to $d$, and half of them are in $PSL(2,q)$. It was proved by Meagher and Spiga  \cite{Karen1} that if $1-d$ is a square in $\mathbb{F}_q^*$ then  $(q-1)/2$ of these elements are derangements. On the other hand,  if $1-d$ is not a square then $(q+1)/2$ of these elements are derangements. 

First, assume that $1-d$ is a square.  We can divide the $(q-1)/2$ elements of $PSL(2,q)$ sending $0$ to $\infty$ and $1$ to $d$ into three categories:
\begin{itemize}
 \item $2$ fix just one point.
 \item $\displaystyle \frac{1}{4} \sum_{ x \in \mathbb{F}_q^*, x \neq 1, -1 } (1 + \phi((x + x^{-1})^2 - 4d))$ fix exactly two points. 
 \item $\displaystyle \frac{q-5}{4} - \frac{1}{4} \sum_{ x \in \mathbb{F}_q^*}  \phi((x + x^{-1})^2 - 4d)$ are derangements.
 \end{itemize} 

A similar analysis can be carried out when $1-d$ is not a square. Specifically,  from the $(q-1)/2$ elements of $PSL(2,q)$ sending $0$ to $\infty$ and $1$ to $d$, 
\begin{itemize}
 \item There are no elements fixing  exactly one point.
 \item $\displaystyle \frac{1}{4} \sum_{ x \in \mathbb{F}_q^*,  x \neq 1, -1 } (1 + \phi((x + x^{-1})^2 - 4d))$ fix two points. 
 \item $\displaystyle \frac{q-1}{4} - \frac{1}{4} \sum_{ x \in \mathbb{F}_q^*}  \phi((x + x^{-1})^2 - 4d)$ are derangements.
\end{itemize}

Putting all the above remarks together and assuming that $1-d$ is a square we obtain,
\begin{eqnarray*}
\sum_{0^g = \infty, 1^g=d} \psi_{-1}(g) & = & \frac{1}{4} \sum_{ x \in \mathbb{F}_q^*, x \neq 1, -1 } (1 + \phi((x + x^{-1})^2 - 4d))\\
   & & - \left(  \frac{q-1}{2} -2 - \frac{1}{4} \sum_{x \in \mathbb{F}_q^*, x \neq 1, -1 } (1 + \phi((x + x^{-1})^2 - 4d)) \right)\\
  & & - \left(  \frac{q-5}{4} - \frac{1}{4} \sum_{ x \in \mathbb{F}_q^*}  \phi((x + x^{-1})^2 - 4d) \right)\\
  & & + \left( \frac{q-1}{2} - \frac{q-5}{4} + \frac{1}{4} \sum_{ x \in \mathbb{F}_q^*}  \phi((x + x^{-1})^2 - 4d) \right) \\
& = & 2 + \sum_{x \in \mathbb{F}_q^*} \phi((x + x^{-1})^2- 4d )\\
 & = &  2 + \sum_{x \in \mathbb{F}_q^*} \phi(x^2 - 2(2d-1)x + 1)(1 + \phi(x))\\
 & = & q P_{\phi} (2d-1).
\end{eqnarray*}
Here the last equality above follows from Equation (\ref{ecu_extra1}).

The case where $(1-d)$ is not a square can be treated by similar computations. We omit the details.
\end{proof}

%%%%%%%%%%%%%%%%%%%%%%%%%%%%%%%%%%
\section{The restriction of $T_N$ onto $V_{\psi_{-1}}$, $V_{\nu_{\gamma}}$ and $V_{\eta_{\beta}}$ }
%%%%%%%%%%%%%%%%%%%%%%%%%%%%%%%%%%

In this section, we study the restriction of $T_N$ onto the irreducible constituents, $V_{\psi_{-1}}$, $\{V_{\nu_{\gamma}} \}_{\gamma \in \Gamma}$ and $\{V_{\eta_{\beta}}\}_{\beta \in B}$, of $V$.  We start with a technical lemma that will be useful for studying the character sums $T_{N,\chi}$ with $\chi \in  \{ \psi_{-1}, \{\eta_{\beta} \}_{\beta \in B}, \{ \nu_{\gamma}\}_{\gamma \in \Gamma} \}$.
\begin{lemma}\label{lemma_extra}
Let $i$ be an element of  $\mathbb{F}_{q^2}^* \setminus \mathbb{F}_q^*$ such that $i^2\in \mathbb{F}_q^*$. Then for all $\gamma \in \Gamma$,
\[
T_{N, \nu_{\gamma}} = \frac{(q-1)}{4} \left[ q^2 - 3q - \left( q + 1 \right)\gamma(-1) \phi(-1) - q^2 \sum_{ b\in \mathbb{F}_q^*, b \neq 1}  P_{\gamma}(2b^h-1) P_{\phi}(2b-1) \right].  
\]
Also, for all $\beta \in B$,
\[
T_{N,\eta_{\beta}} =  \frac{(q-1)}{4} \left[ q^2  + q + \left( q + 1 \right)\beta(i) \phi(-1) + q^2\sum_{ b\in \mathbb{F}_q^*, b \neq 1}  R_{\beta}(2b^h-1) P_{\phi}(2b-1) \right],
\]
and
\[
T_{N,\psi_{-1}} = \frac{(q-1)}{4} \left[ q^2 - 2 q  -3 - q^2 \sum_{ b\in \mathbb{F}_q^*, b \neq 1}  P_{\phi}(2b^h-1) P_{\phi}(2b-1) \right]. 
\]
\end{lemma}

\begin{proof}
We will prove that the expression for $T_{N, \nu_{\gamma}}$ holds for every $\gamma \in \Gamma$. The proofs for the characters sums $T_{N,\eta_{\beta}}$ and $T_{N,\psi_{-1}} $ are similar; we omit those details.

First, assume that $q \equiv 1 \mbox{ mod }4$. It follows from Lemma \ref{lemma7} that 
\begin{eqnarray*}
T_{N,\nu_{\gamma}}  & = &\frac{(q-1)^3}{4} - \frac{q-1}{2} \sum_{0^g = \infty, \infty^g=0} \nu_{\gamma}(g^{-1}) + (q-1)  \sum_{\substack{b \in \mathbb{F}_q^* \\ b \neq 1}}  \left[ \sum_{0^g=\infty,  1^g =b^h}  \nu_{\gamma}(g^{-1}) \right] N_{(0,\infty),(1,b)}\\
 & = & \frac{(q-1)^3}{4} - \frac{(q-1)^2}{2} \gamma(-1)  + (q-1)  \sum_{\substack{b \in \mathbb{F}_q^* \\ b \neq 1}}  \left[ \sum_{0^g=\infty,  1^g =b^h}  \nu_{\gamma}(g^{-1}) \right] N_{(0,\infty),(1,b)}, 
\end{eqnarray*}
where for the last equality we have applied Lemma \ref{lemma8}. Also, recall that  $h \in PGL(2,q)$ is the unique element sending $0$ to $0$, $1$ to $\infty$ and $\infty$ to $1$. 

Let us define
\[
	S :=  \sum_{\substack{b \in \mathbb{F}_q^* \\ b \neq 1}}  \left[ \sum_{0^g=\infty,  1^g =b^h}  \nu_{\gamma}(g^{-1}) \right] N_{(0,\infty),(1,b)}.
\]
Applying Corollary \ref{usefulcor} and  Lemma \ref{lemma9} we obtain
\begin{eqnarray*}
S & = &  \sum_{\substack{b \in \mathbb{F}_q^* \\ b \neq 1}} qP_{\gamma}(2b^h-1) \left(  \frac{q-1}{4} - \frac{\phi(1-b)}{2} - \frac{1}{4} P_{\phi}(2b-1)  \right)\\
   & = & \frac{q(q-1)}{4} \sum_{\substack{b \in \mathbb{F}_q^* \\ b \neq 1}}  P_{\gamma}(2b^h-1) - \frac{q}{2} \sum_{\substack{b \in \mathbb{F}_q^* \\ b \neq 1}} \phi(1-b) P_{\gamma}(2b^h-1) - \frac{q^2}{4}  \sum_{\substack{b \in \mathbb{F}_q^* \\ b \neq 1}} P_{\gamma}(2b^h-1)P_{\phi}(2b-1).
\end{eqnarray*}
We now simplify the first two character sums in the above expression for $S$.

The following computation uses the connection between Legendre sums and hypergeometric sums given by Lemma  \ref{lemma13}. We have
\begin{eqnarray*}
\sum_{\substack{b \in \mathbb{F}_q^* \\ b \neq 1}}  P_{\gamma}(2b^h-1) & = &  \sum_{\substack{a \in \mathbb{F}_q \\ a \neq \pm 1}} P_{\gamma}(a) \\
     & = &   \sum_{\substack{a \in \mathbb{F}_q \\ a \neq \pm 1}} \hgq{\gamma}{\gamma^{-1}}{\epsilon}{\frac{1-a}{2};q}.
\end{eqnarray*}
Now, using Greene's definition of hypergeometric sums given in Equation (\ref{ecu12}) we get
\begin{eqnarray*}
\sum_{\substack{b \in \mathbb{F}_q^* \\ b \neq 1}}  P_{\gamma}(2b^h-1) & = &   \frac{\gamma^{-1}(-1)}{q} \sum_{\substack{a \in \mathbb{F}_q \\ a \neq \pm 1}} \sum_{x \in \F_q} \gamma^{-1}(x) \gamma(1-x) \gamma^{-1}\left(1-\frac{1}{2}(1-a)x \right) \\
     & = & \frac{\gamma^{-1}(-1)}{q}  \sum_{x \in \F_q^*} \gamma^{-1}(x) \gamma(1-x)  \sum_{\substack{a \in \mathbb{F}_q \\ a \neq \pm 1}} \gamma^{-1}\left(1-\frac{1}{2}(1-a)x \right) \\
     & = & \frac{\gamma^{-1}(-1)}{q}  \sum_{x \in \F_q^*} \gamma^{-1}(x) \gamma(1-x)  (-1 - \gamma^{-1}(1-x) )\\
     & = & \frac{1}{q} (1 + \gamma(-1)).
\end{eqnarray*} 

On the other hand, to compute the second sum we use the definition of Legendre sums given in Definition \ref{def1} and noting that $\phi(-1)=1$ when $q\equiv 1 \mod 4$, 
\begin{eqnarray*}
    \sum_{\substack{b \in \mathbb{F}_q^* \\ b \neq 1}} \phi(1-b) P_{\gamma}(2b^h-1) & = & \frac{1}{q}  \sum_{\substack{b \in \mathbb{F}_q^* \\ b \neq 1}} \phi(1-b)  \sum_{x \in \F_q^*} \gamma(x) \phi(1+(2-4b^h)x + x^2 ) \\
    & = & \frac{1}{q} \sum_{x \in \F_q^*} \gamma(x)  \sum_{\substack{b \in \mathbb{F}_q^* \\ b \neq 1}} \phi((x+1)^2 - 4b^hx)\phi(b-1) \\
    & = & \frac{1}{q} \sum_{x \in \F_q^*} \gamma(x)  \sum_{\substack{b \in \mathbb{F}_q^* \\ b \neq 1}} \phi( (x-1)^2 b - (x+1)^2) \\&=&  \frac{1}{q} \sum_{x \in \F_q^*, x\neq 1} \gamma(x)  \sum_{\substack{b \in \mathbb{F}_q^* \\ b \neq 1}} \phi( (x-1)^2 b - (x+1)^2) +\frac 1q   \sum_{\substack{b \in \mathbb{F}_q^* \\ b \neq 1}} \phi(-4)  \\
    &=&  \frac{1}{q} \sum_{x \in \F_q^*, x\neq 1} \gamma(x) (-\phi(-4x)-\phi(-(x+1)^2)) +\frac {q-2}q    \\
    & = & 1 + \frac{1}{q} \gamma(-1). \bk
\end{eqnarray*}  

Putting all the above results together we have
\[
S= - \frac{(q-1)}{4} + \frac{(q-3)}{4}\gamma(-1) - \frac{q^2}{4} \sum_{\substack{b \in \mathbb{F}_q^* \\ b \neq 1}} P_{\gamma}(2b^h-1)P_{\phi}(2b-1),
\] 
and plugging in $S$ into the expression for $T_{N,\nu_{\gamma}} $ we obtain
\[
T_{N,\nu_{\gamma}} = \frac{q-1}{4} \left[ q^2 - 3q - (q-1)\gamma(-1) - q^2 \sum_{\substack{b \in \mathbb{F}_q^* \\ b \neq 1}} P_{\gamma}(2b^h-1)P_{\phi}(2b-1)\right].
\]
The computations for the case  $q \equiv 3 \mbox{ mod }4$ are very similar.  In fact, the following expression is obtained for $T_{N,\nu_{\gamma}} $ assuming that $q \equiv 3 \mbox{ mod }4$,
\[
T_{N,\nu_{\gamma}} = \frac{q-1}{4} \left[ q^2 - 3q + (q-1)\gamma(-1) - q^2 \sum_{\substack{b \in \mathbb{F}_q^* \\ b \neq 1}} P_{\gamma}(2b^h-1)P_{\phi}(2b-1)\right].
\]
Finally, note that $\phi(-1)=1$ when $q \equiv 1 \mbox{ mod }4$ and $\phi(-1)=-1$ when $q \equiv 3 \mbox{ mod }4$. This fact completes the proof of the Lemma. 
\end{proof}

From Schur's Lemma we know that the restriction of $T_N$ onto any irreducible module is an isomorphism or the zero map.  The next theorem shows that the restriction of $T_N$ onto $V_{\eta_{\beta}}$ is a $PGL(2,q)$-module isomorphism for every $\beta \in B$.

%%%%%%%I moved $f$ forward%%%%%%%%%%
For the proofs below, we will need the following function in $\ell^2(\mathbb{F}_q,m) $,
\[
\begin{array}{cccc}
f : & \mathbb{F}_q & \rightarrow & \mathbb{C} \\
    & x                    & \mapsto      & \phi(1-x)P_{\phi}(x)
\end{array}
\]
Note that the norm of $f$ is closely related to the norm of $P_{\phi}$,
\[
\Vert f \Vert^2 = \sum_{x \in \mathbb{F}_q} f(x)^2m(x) = \sum_{\substack{ x \in \mathbb{F}_q \\  x \neq 1 } } P_{\phi}(x)^2m(x) = \Vert P_{\phi}\Vert^2 - \frac{q+1}{q^2} = 1-\frac{1}{q}- \frac{2}{q^2},
\]
where we have used Lemma \ref{lemma20} in the last equality.  

\begin{theorem}\label{teo2}
For every $\beta \in B$ we have 
\[
T_N(V_{\eta_{\beta}}) \cong V_{\eta_{\beta}}.
\]
\end{theorem}

\begin{proof}
It suffices to show that $T_{N,\eta_{\beta}} \neq 0$ for all $\beta \in B$.  From Lemma \ref{lemma_extra} it follows that
\begin{equation}\label{psl_exp1}
T_{N,\eta_{\beta}} =  \frac{(q-1)}{4} \left[ q^2  + q + \left( q + 1 \right)\beta(i) \phi(-1) + q^2\sum_{ b\in \mathbb{F}_q^*, b \neq 1}  R_{\beta}(2b^h-1) P_{\phi}(2b-1) \right],
\end{equation}
where $i\in \F_{q^2}^*\setminus \F_q^*$ such that $i^2\in \F_q^*$.  We will show that the expression on the right hand side of Equation (\ref{psl_exp1}) is not equal to zero. 

We claim that the character sum 
\begin{equation}\label{ecu11}
\sum_{ b\in \mathbb{F}_q^*, b \neq 1}  R_{\beta}(2b^h-1) P_{\phi}(2b-1) 
\end{equation}
can be expressed in terms of the function $f$. Recall that $h$ is the unique element in $PGL(2,q)$ sending $0$ to $0$, $1$ to  $\infty$ and $\infty$ to $1$. Hence, if $ b\in \mathbb{F}_q^* $ and  $b \neq 1$ then $b^h \neq 0, 1, \infty$. Moreover, we have the following formula for $b^h$ when $b\in \mathbb{F}_q^*$  and $b \neq 1$,
\[
b^h= \frac{b}{b-1}
\] 
which implies that $(b^h)^h=b$ for any $b \in \mathbb{F}_q$. Thus, we can rewrite the sum in (\ref{ecu11}) as,
\[
\sum_{ b\in \mathbb{F}_q^*, b \neq 1}  R_{\beta}(2b^h-1) P_{\phi}(2b-1)  = \sum_{ b\in \mathbb{F}_q^*, b \neq 1}  P_{\phi}(2b^h-1) R_{\beta}(2b-1). 
\]

Using the relation between Legendre sums and hypergeometric sums given by Lemma \ref{lemma13}  and the transformation formula in Lemma \ref{lemma12}, the following expression for $P_{\phi}(2b^h-1)$ is obtained 
\[
P_{\phi}(2b^h -1) =\hgq{\phi}{\phi}{\epsilon}{\frac{1}{1-b};q} = \phi(1-b) \hgq{\phi}{\phi}{\epsilon}{1-b;q}  = \phi(1-b)P_{\phi}(2b-1), 
\] 
for $b \in \mathbb{F}_q$, $b \neq 0, 1$. Putting all the above remarks together we conclude that
\begin{eqnarray*}
\sum_{ b\in \mathbb{F}_q^*, b \neq 1}  R_{\beta}(2b^h-1) P_{\phi}(2b-1) & = & \sum_{ b\in \mathbb{F}_q^*, b \neq 1}  \phi(1-b) P_{\phi}(2b-1) R_{\beta}(2b-1) \\
& = &  \phi(2) \sum_{ x\in \mathbb{F}_q, x \neq \pm 1}  \phi(1-x) P_{\phi}(x) R_{\beta}(x)\\
& = & \phi(2) \left(1 +\frac{1}{q} \right)^{1/2} \langle f , R_{\beta}' \rangle - (q+1) \frac{\beta(i)\phi(-1) }{q^2}
\end{eqnarray*}
where $i$ is an element of $\mathbb{F}_{q^2}^*\setminus \F_q^*$ such that $i^2 \in \mathbb{F}_q^*$.
Therefore, plugging in the above expression into Equation (\ref{psl_exp1}), we can also express $T_{N,\eta_{\beta}}$ in terms of the function $f$,
\begin{equation}\label{ecu12}
T_{N,\eta_{\beta}}= \frac{q^2(q-1)}{4} \left[  1+ \frac{1}{q} + \phi(2) \left(1 +\frac{1}{q} \right)^{1/2} \langle f , R_{\beta}' \rangle  \right].
\end{equation}
Note that Equation (\ref{ecu12}) implies that if $|\langle f , R_{\beta}' \rangle| \leq 1$ then $T_{N,\eta_{\beta}} \neq 0$. We claim that $|\langle f , R_{\beta}' \rangle| \leq 1$ for every $\beta \in B$;  note that the theorem follows from the validity of this claim.

Recall that $\{ P_{\epsilon}', P_{\phi}', P_{\gamma}', R_{\beta}' : \mbox{ } \gamma \in  \Gamma, \beta \in B \}$ is an orthonormal basis of $\ell^2(\mathbb{F}_q,m)$. Thus, we can express $f$ in terms of this orthonormal basis,
\[
f= \langle f, P_{\epsilon}'\rangle P_{\epsilon}' + \langle f, P_{\phi}'\rangle P_{\phi}'  +\sum_{\gamma} \langle f, P_{\gamma}' \rangle P_{\gamma}' + \sum_{\beta} \langle f, R_{\beta}'\rangle R_{\beta}'. 
\]
Analogously, the squared norm of $f$ can also be expressed in terms of this orthonormal basis,
\[
\Vert f \Vert^2 = \langle f, P_{\epsilon}'\rangle^2 + \langle f, P_{\phi}'\rangle^2 + \sum_{\gamma} \langle f, P_{\gamma}' \rangle^2 + \sum_{\beta} \langle f, R_{\beta}'\rangle^2,
\]
where we have used the fact the coefficients in the expansion of $f$ are all real (cf. Lemma~\ref{lemma21}). 

On the other hand, we know that the squared norm of $f$ is $1-1/q-2/q^2$. This implies that the square of every coefficient of the form  $\langle f, g\rangle$ is less than 1 for all $g \in \{ P_{\epsilon}', P_{\phi}', P_{\gamma}', R_{\beta}' : \mbox{ } \gamma \in  \Gamma, \beta \in B \}$. In particular, $\langle f, R_{\beta}'\rangle^2 \leq 1-1/q-2/q^2 $ for all $\beta \in B$. Thus, our claim is proved.
\end{proof}

Unfortunately, the argument used in the proof of Theorem \ref{teo2} cannot be applied to show that the restriction of $T_N$ onto the irreducible module  $V_{\psi_{-1}}$ is a $PGL(2,q)$-module isomorphism. To deal with this case we exploit the connection between Legendre sums and Hypergeometric sums shown by Kable in \cite{Kable}. 

%%%%%%%Separated the following claim as a Lemma%%%%%%

\begin{lemma}\label{lem:<f,P_gamma>}Let $\gamma$ be a nontrivial multiplicative character of $\F_q$. Then $$\phi(2) q^2 \langle f , P_{\gamma} \rangle=
q^3 \pFFq{4}{3}{\gamma & \gamma^{-1} & \phi & \phi}{ & \epsilon & \epsilon & \epsilon}{1 ; q} +  \phi(-1)\gamma(-1)q.$$\end{lemma} 
\begin{proof}Applying  Lemmas \ref{lemma15} and \ref{lemma13} we obtain,
\begin{eqnarray*}
\phi(2) q^2 \langle f , P_{\gamma} \rangle  & = & \phi(2) q^2 \sum_{\substack{ x \in \mathbb{F}_q \\ x \neq \pm 1}} \phi(1-x) P_{\phi}(x)P_{\gamma}(x) + q^2 P_{\phi}(-1)P_{\gamma}(-1) m(-1) \\
   & = & q^2 \sum_{\substack{ y \in \mathbb{F}_q^* \\ y \neq 1}}  \phi(y) \hgq{\phi}{\phi}{\epsilon}{y;q} \hgq{\gamma}{\gamma^{-1}}{\epsilon}{y;q} + \phi(-1)\gamma(-1)(q+1) \\
   & = &  q^2 \sum_{y \in \mathbb{F}_q}  \phi(y)  \hgq{\phi}{\phi}{\epsilon}{y;q}\hgq{\gamma}{\gamma^{-1}}{\epsilon}{y;q} +  \phi(-1)\gamma(-1)q \\
   & = &  q^3 \pFFq{4}{3}{\gamma & \gamma^{-1} & \phi & \phi}{ & \epsilon & \epsilon & \epsilon}{1 ; q} +  \phi(-1)\gamma(-1)q.
\end{eqnarray*}
\end{proof}

\begin{theorem}\label{teo6}
If $q \geq 7$ then,
\[
T_N(V_{\psi_{-1}}) \cong V_{\psi_{-1}}.
\]
\end{theorem}

\begin{proof}
It suffices to show that  $T_{N,\psi_{-1}} \neq 0$. It follows from Lemma \ref{lemma_extra} that
\[
T_{N,\psi_{-1}} = \frac{(q-1)}{4} \left[ q^2 - 2 q  -3 - q^2 \sum_{ b\in \mathbb{F}_q^*, b \neq 1}  P_{\phi}(2b^h-1) P_{\phi}(2b-1) \right].  
\]

Let $f$ be the function in $\ell^2(\mathbb{F}_q,m)$ defined before the statement of Theorem \ref{teo2}. By Lemmas \ref{lemma12} and \ref{lemma13}  we see that the sum
\[
\sum_{ b\in \mathbb{F}_q^*, b \neq 1}  P_{\phi}(2b^h-1) P_{\phi}(2b-1) 
\]
can be written in terms of the function $f$. In particular, 
\begin{eqnarray*}
\sum_{ b\in \mathbb{F}_q^*, b \neq 1}  P_{\phi}(2b^h-1) P_{\phi}(2b-1) & = & \sum_{ b\in \mathbb{F}_q^*, b \neq 1}  \phi(1-b) P_{\phi}(2b-1) P_{\phi}(2b-1) \\
& = &  \phi(2) \sum_{ x\in \mathbb{F}_q, x \neq \pm 1}  \phi(1-x) P_{\phi}(x) P_{\phi}(x)\\
& = & \phi(2)  \langle f , P_{\phi} \rangle -  \frac{q+1}{q^2}.
\end{eqnarray*}
Thus, $T_{N,\psi_{-1}} $ can be expressed in terms of $f$:
\begin{equation}\label{ecu20}
T_{N,\psi_{-1}} = \frac{(q-1)}{4} \left[  q^2-  q - 2 - \phi(2) q^2 \langle f , P_{\phi} \rangle  \right].
\end{equation}

We claim that $ \phi(2) q^2 \langle f , P_{\phi} \rangle \leq 2q^{3/2}$. This claim together with Equation (\ref{ecu20}) immediately implies that $T_{N,\psi_{-1}} \neq 0$ for every $q \geq 7$. 

To prove our claim we note that the character sum $\phi(2) q^2 \langle f , P_{\phi} \rangle$ can be written in terms of a hypergeometric sum $_4\F_3$.   Letting $\gamma=\phi$ in Lemma \ref{lem:<f,P_gamma>}, $$\phi(2) q^2 \langle f , P_{\phi} \rangle= q^3 \pFFq{4}{3}{\phi & \phi & \phi & \phi}{ & \epsilon & \epsilon & \epsilon}{1 ; q} + q.$$    Therefore, our claim follows directly from the final conclusion of Proposition \ref{prop:15}. 
\end{proof}

To study the restriction of $T_N$ onto $V_{\nu_{\gamma}}$ we consider two cases. First, if $\gamma$ is a character whose order is not equal to three, four or six then we can apply arguments similar to the ones used in the proof of Theorem \ref{teo2} to prove that the restriction is an isomorphism. On the other hand, different ideas have to be used to show that the same result holds when $\gamma$ has order three, four or six. The next theorem deals with these cases.

\begin{theorem}\label{teo3}
Assume that $q \geq 11$. If $\gamma \in \Gamma$ % is a character whose order is not equal to three, four or six 
then
\[
T_N(V_{\nu_{\gamma}}) \cong V_{\nu_{\gamma}}.
\]
\end{theorem}

\begin{proof}

We proceed as we did in the proof of Theorem \ref{teo2}. Thus, to prove this theorem it is enough to show that  $T_{N,\nu_{\gamma}} \neq 0$. It follows from Lemma \ref{lemma_extra} that
\[
T_{N, \nu_{\gamma}} = \frac{(q-1)}{4} \left[ q^2 - 3q - \left( q + 1 \right)\gamma(-1) \phi(-1) - q^2 \sum_{ b\in \mathbb{F}_q^*, b \neq 1}  P_{\gamma}(2b^h-1) P_{\phi}(2b-1) \right].  
\]

%Let $f$ be the function in $\ell^2(\mathbb{F}_q,m)$ defined in the proof of Theorem \ref{teo2}.  

Applying Lemmas \ref{lemma12} and \ref{lemma13} it is possible to write the sum of products of Legendre sums in terms of the function $f$. In fact,
\begin{eqnarray*}
\sum_{ b\in \mathbb{F}_q^*, b \neq 1}  P_{\gamma}(2b^h-1) P_{\phi}(2b-1)  & = & \phi(2) \left( 1 -\frac{1}{q} \right)^{1/2} \langle f , P_{\gamma}' \rangle - (q+1) \frac{\gamma(-1)\phi(-1) }{q^2}. 
\end{eqnarray*}
Therefore,  for every $\gamma \in \Gamma$  we have 
\begin{equation}\label{ecu18_2}
 T_{N, \nu_{\gamma}} = \frac{q^2(q-1)}{4}  \left[ 1- \frac{3}{q} -  \phi(2) \left( 1 -\frac{1}{q} \right)^{1/2} \langle f , P_{\gamma}' \rangle  \right].
\end{equation}

Recall that
\begin{equation}\label{ecu18}
\Vert f \Vert^2=\langle f, P_{\epsilon}'\rangle^2 + \langle f, P_{\phi}'\rangle^2 + \sum_{\gamma} \langle f, P_{\gamma} \rangle^2 + \sum_{\beta} \langle f, R_{\beta}'\rangle^2  = 1-\frac{1}{q} - \frac{2}{q^2},
\end{equation}
where $\{ P_{\epsilon}', P_{\phi}', P_{\gamma}', R_{\beta}' : \mbox{ } \gamma \in  \Gamma, \beta \in B \}$ is an orthonormal basis of $\ell^2(\mathbb{F}_q,m)$. Equation (\ref{ecu18}) implies that  at most one of the coefficients $ \langle f, g \rangle$ with $g \in \{ P_{\epsilon}', P_{\phi}', P_{\gamma}', R_{\beta}' : \mbox{ } \gamma \in  \Gamma, \beta \in B \}$ 
 can be close to $1$. On the other hand, it is clear from (\ref{ecu18_2}) that $ T_{N, \nu_{\gamma}} = 0$ if and only if the coefficient $\langle f, P_{\gamma}' \rangle$ is close to $1$.  

To prove the theorem we proceed by contradiction. Assume that there exists $\gamma \in \Gamma$ such that $ T_{N, \nu_{\gamma}} = 0$. Hence, it follows from equation (\ref{ecu18_2}) that 
\begin{equation}\label{ecu18_3}
\langle f , P_{\gamma}' \rangle^2 =1 -\frac{5}{q} + \frac{4}{q(q-1)}.
\end{equation}

Let $\mbox{Gal}(\Q(\zeta_{q-1})/\Q )$ be the Galois group where $\zeta_{q-1}$ is a primitive $(q-1)$-th root of the unity. If $\gamma$ is a nontrivial character whose order is not equal to three, four or six,   there exists $\sigma \in \mbox{Gal}(\Q(\zeta_{q-1})/\Q )$ such that $\gamma^{\sigma} \neq \gamma$ and $\gamma^{\sigma} \neq \gamma^{-1}$. Now, applying the Galois automorphism $\sigma$ to both sides of  (\ref{ecu18_3}) we conclude that
\begin{eqnarray*}
\sigma \left( \langle f , P_{\gamma}' \rangle^2 \right)  & = & \sigma\left(1 -\frac{5}{q} + \frac{4}{q(q-1)}\right)\\
\langle f , P_{\gamma^{\sigma}}' \rangle^2  & = & 1 -\frac{5}{q} + \frac{4}{q(q-1)}.
\end{eqnarray*}
Thus, $ \langle f , P_{\gamma}' \rangle^2 $ and $\langle f , P_{\gamma^{\sigma}}' \rangle^2$ are equal to $1 -\frac{5}{q} + \frac{4}{q(q-1)}$ which is a contradiction because at most one of the coefficients $ \langle f, g \rangle$ with $g \in \{ P_{\epsilon}', P_{\phi}', P_{\gamma}', R_{\beta}' : \mbox{ } \gamma \in  \Gamma, \beta \in B \}$ 
 can be close to $1$.  Assume now   $\gamma \in \Gamma$ is a character of order $3$, $4$ or  $6$.   From equation (\ref{ecu18_2}) we get the following expression for $T_{N,\nu_{\gamma}}$,
\[
T_{N, \nu_{\gamma}} = \frac{(q-1)}{4}  \left[  q^2  - 3q -  \phi(2) q^2 \langle f , P_{\gamma} \rangle  \right].
\]
By Lemma \ref{lem:<f,P_gamma>}, $$\phi(2) q^2 \langle f , P_{\gamma} \rangle =q^3 \pFFq{4}{3}{\gamma & \gamma^{-1} & \phi & \phi}{ & \epsilon & \epsilon & \epsilon}{1 ; q} +  \phi(-1)\gamma(-1)q.$$ 
Now applying Proposition \ref{prop:15}, we conclude that $T_{N\nu_{\gamma}}\neq 0$.
\end{proof}

Finally, we are ready to prove Theorem \ref{psl_teo2}.

\begin{proof}[Proof of Theorem \ref{psl_teo2}]
Recall that in Section \ref{psl_im_TN}  we proved the following lower and upper bounds on the rank of the derangement matrix $M$ of $PSL(2,q)$ acting on $PG(1,q)$,
\begin{equation}
	\sum_{\{ \chi : \mbox{ }T_{N,\chi} \neq 0 \}} \dim(V_{\chi}) \leq \mbox{rank}(M) \leq q(q-1).
\end{equation}
These bounds imply that if $T_{N,\chi}$ is not zero for every  $\chi \in  \{ \lambda_1, \psi_{-1}, \{\eta_{\beta} \}_{\beta \in B}, \{ \nu_{\gamma}\}_{\gamma \in \Gamma} \}$  then the rank of $M$ is $q(q-1)$.

 If $q \geq 11$ then it follows from Theorems \ref{teo2}, \ref{teo6} and \ref{teo3} that $T_{N,\chi}\neq 0$ for all $\chi \in  \{ \lambda_1, \psi_{-1}, \{\eta_{\beta} \}_{\beta \in B}, \{ \nu_{\gamma}\}_{\gamma \in \Gamma} \}$. Furthermore, for each odd prime power $q$, $3<q<11$, we use a computer to check that the rank of $M$ is exactly $q(q-1)$.
\end{proof}

%Now, assume that $q-1$ is divisible by $3$, $4$ or $6$. If $q$ is big enough then it follows from Theorems %\ref{teo2}, \ref{teo6} and \ref{teo4} that $T_{N,\chi}\neq 0$ for every $\chi \in  \{ \lambda_1, \psi_{-1}, %\{\eta_{\beta} \}_{\beta \in B}, \{ \nu_{\gamma}\}_{\gamma \in \Gamma} \}$.

\section{Conclusions}

In this paper we consider the natural right action of $PSL(2,q)$ on $PG(1,q)$, where $q$ is an odd prime power.  Using the eigenvalue method,  it was proved in \cite{Karen1, Karen3}  that the maximum size of an intersecting family in $PSL(2,q)$ is $q(q-1)/2$. Meagher and Spiga \cite{Karen1} conjectured that the cosets of point stabilizers are the only intersecting families of maximum size in $PSL(2,q)$, when $q>3$ is an odd prime power.  Here, we prove their conjecture in the affirmative using tools from representation theory of $PGL(2,q)$ and deep results from number theory. 

For future research, one could consider the stability problem concerning intersecting families of $PSL(2,q)$. To present this problem we introduce the notion of stability.

Let $X$ be a finite set and $G$ a finite group acting on $X$. Recall that a subset $S$ of $G$ is said to be an  intersecting family if for any $g_1,g_2 \in S$ there exists an element $x\in X$ such that $x^{g_1}= x^{g_2}$. We will refer to intersecting families of maximum size as {\it extremal families}. Moreover, intersecting families whose sizes are close to the maximum are called {\it almost extremal families}. We say that the extremal families of a group $G$ acting on $X$ are {\it stable} if almost extremal families are similar in structure to the extremal ones.

The stability of intersecting families has been studied during the past few years (cf. \cite{Ellis5, eff, Plaza}).  Consider the action of $S_n$ on $[n]$. As was remarked in the introduction, the size of extremal families in $S_n$ is $(n-1)!$ and every extremal family is a coset of a point stabilizer. Furthermore, the stability of extremal families in $S_n$ was established by Ellis \cite{Ellis5}, who  proved that for any $\epsilon >0$ and $n > N(\epsilon)$, any intersecting family of size at least $(1-1/e +\epsilon) (n-1)!$ must be strictly contained in an extremal family. Analogously, the same problems were solved for the group $PGL(2,q)$ acting on $PG(1,q)$. In fact, the size of extremal families in $PGL(2,q)$ is $q(q-1)$ and every extremal family is a coset of a point stabilizer. Recently, in \cite{Plaza} it was proved that the extremal families in $PGL(2,q)$ are stable.
 
We conjecture that the extremal families in $PSL(2,q)$ are also stable. The precise statement is given below.

\begin{conjecture}
Let $S$ be an intersecting family in $PSL(2,q)$ with $q>3$ an odd prime power. Then there exists $\delta > 0$ such that if $|S| \geq (1-\delta) q(q-1)/2$ then $S$ is contained within a coset of a point stabilizer.
\end{conjecture} 

\section*{Acknowledgment}
The authors would like to thank the reviewers for their helpful comments.

\end{document}